\documentclass[11pt]{amsart}
\usepackage{hyperref}
\usepackage{amsxtra}
\usepackage{amssymb}
\usepackage{comment}
\usepackage{mathrsfs} 
\usepackage{enumerate}
\usepackage{array}
\usepackage{xcolor}
\usepackage{cancel}
\usepackage{tikz}
\theoremstyle{plain}

\newtheorem{theorem}{Theorem}[section]
\newtheorem{lemma}[theorem]{Lemma}
\newtheorem{definition-theorem}[theorem]{Definition-Theorem}
\newtheorem{proposition}[theorem]{Proposition}

\newtheorem{definition}[theorem]{Definition}
\newtheorem{remark}[theorem]{Remark}

\theoremstyle{definition}
\newtheorem{example}[theorem]{Example}

\newcommand{\bbr}{\mathbb{R}}
\newcommand{\bba}{\mathbb{A}}
\newcommand{\bbc}{\mathbb{C}}
\newcommand{\bbz}{\mathbb{Z}}
\newcommand{\bbn}{\mathbb{N}}

\newcommand{\ord}{\mathrm{ord}}
\newcommand{\wt}{\widetilde}

\DeclareMathOperator{\Grad}{{\mathrm{Gr}^{\mathrm{ad}}}}

\newcommand{\vocab}[1]{\textit{#1}}
\DeclareMathOperator{\ad}{\text{ad}}
\newcommand{\ol}[1]{\overline{#1}}
\DeclareMathOperator \Span { {\mathrm{Span}} }

\renewcommand{\Im}{\mathrm{Im}}
\DeclareMathOperator{\LM}{\mathrm{LM}}
\DeclareMathOperator{\vspan}{\mathrm{Span}}
\DeclareMathOperator{\Proj}{\mathrm{Proj}}

\DeclareMathOperator{\spec}{\mathrm{Spec}}
\DeclareMathOperator{\Hom}{\mathrm{Hom}}
\DeclareMathOperator{\const}{\mathrm{const}}
\DeclareMathOperator{\Be}{\mathrm{Be}}

\newcommand{\ub}{h}
\newcommand{\diff}{\backslash}
\usepackage{amsmath,calligra,mathrsfs}
\DeclareMathOperator{\shom}{\mathscr{H}\text{\kern -3pt {\calligra\large om}}\,}

\begin{document}
\title[Reflective prolate-spheroidal operators and the adelic Grassmannian]
{Reflective prolate-spheroidal operators and \\
the adelic Grassmannian}
\author[W. Riley Casper]{W. Riley Casper}
\address{
Dept of Mathematics \\
Louisiana State University \\
Baton Rouge, LA 70803 \\
U.S.A.
}
\email{wcasper1@lsu.edu}

\author[F. Alberto Gr\"unbaum]{F. Alberto Gr\"unbaum}
\address{
Dept of Mathematics \\
University of California, Berkeley \\
Berkeley, CA 94720 \\
U.S.A.
}
\email{grunbaum@math.berkeley.edu}

\author[Milen Yakimov]{Milen Yakimov}
\address{
Dept of Mathematics \\
Louisiana State University \\
Baton Rouge, LA 70803 \\
U.S.A.
}
\email{yakimov@math.lsu.edu}

\author[Ignacio Zurri\'an]{Ignacio Zurri\'an}
\address{
Dept of Mathematics \\
Universidad Nacional de C\'ordoba, C\'ordoba, X5016HUA, Argentina
}
\email{zurrian@famaf.unc.edu.ar}
\date{\today} 
\begin{abstract}
Beginning with the work of Landau, Pollak and Slepian in the 1960s on time-band limiting, commuting pairs of integral and differential operators have played a key role in signal processing, random matrix theory, and integrable systems. Previously, such pairs were constructed by ad hoc methods, which 
essentially worked because a commuting operator of low order could be found by a direct calculation. 

We describe a general approach to these problems 
that proves that every point $W$ of Wilson's infinite dimensional adelic Grassmannian $\Grad$ gives rise to an integral operator $T_W$, acting on $L^2(\Gamma)$ for a contour 
$\Gamma \subset \bbc$, which reflects a differential operator with rational coefficients $R(z, \partial_z)$ in the sense that $R(-z,-\partial_z) \circ  T_W = T_W \circ R(w, \partial_w)$
on a dense subset of $L^2(\Gamma)$. By using analytic methods and methods from integrable systems, we show that the reflected differential operator can be constructed 
from the Fourier algebra of the associated bispectral function $\psi_W(x,z)$. The exact size of this algebra with respect to a bifiltration is in turn determined using 
algebro-geometric methods. Intrinsic properties of four involutions of the adelic Grassmannian naturally lead us to consider the reflecting property above in place of plain commutativity. 
Furthermore, we prove that the time-band limited operators of the generalized Laplace transforms with kernels given by the rank one bispectral functions $\psi_W(x,-z)$
always reflect a differential operator.  A $90^\circ$ rotation argument is used to prove that the time-band limited operators of the generalized Fourier transforms with kernels
$\psi_W(x, iz)$ admit a commuting differential operator. These methods produce vast collections of integral operators with prolate-spheroidal properties, 
associated to the wave functions of all rational solutions of the KP hierarchy vanishing at infinity, introduced by Krichever
in the late 1970s. 
\end{abstract}
\maketitle
\section{Background}
\subsection{The prolate-spheroidal phenomenon}
\label{1.1}
In his foundational work \cite{Shannon1} on communication theory, Claude Shannon posed the problem of finding
the best information that one can infer about a signal $f(t)$ which is supported in the interval $[-\tau,\tau]$ from the knowledge of its frequencies in the interval $[-\kappa,\kappa]$.
This requires determining the singular value decomposition of the finite Fourier transform
$$
(Ef)(z)=\int_{-\tau}^\tau e^{izx}f(x) dx,\quad {z\in[-\kappa,\kappa]},
$$
that is, the eigenfunctions of the integral operator
$$
(EE^*f)(z)=2\int_{-\kappa}^\kappa\frac{\sin \tau(z-w)}{z-w}f(w) dw, \; {z\in[-\kappa,\kappa]}.
$$
This is a numerically extremely unstable problem.  Landau, Pollak and Slepian  \cite{SP, LP1} bypassed this issue based on the property that
$$
R(z, \partial_z)= \partial_z (\kappa^2-z^2)\partial_z -\tau^2 z^2
$$ 
commutes with $E E^*$. The differential operator $R(z,\partial_z)$ is the ``radial part'' of the Laplacian in prolate-spheroidal coordinates, 
and (the joint) eigenfunctions of it and the integral operator $E E^*$ are known as the prolate-spheroidal functions, which can be computed numerically in very stable ways. 
The commuting pair can be traced back to Bateman \cite[Eqs. 38-41 plus differentiation]{bateman} and the classical text by Ince \cite[p. 201]{Ince}.
Mehta  \cite{Mehta} independently discovered and used it to analyze the Fredholm determinant of 
$E E^*$ in relation to asymptotic problems for random matrices.
An in-depth study of the numerical properties of prolate-spheroidal functions (the joint eigenfunctions of $E E^*$ and $R(z, \partial_z)$) 
was carried out in \cite{ORX}.

Slepian \cite{S} extended the time-band limiting analysis to $n$-dimensions by passing to polar coordinates and relying on a different commuting pair. 
He proved that the integral operator
\[
(\mathcal Ef)(z) = \int_0^1 J_N(czw)\sqrt{czw}f(w)dw
\]
acting on a subspace of $L^2(0, 1; dw)$ with appropriate boundary conditions commutes with 
\[
\partial_z (1-z^2)\partial_z - c^2z^2+\frac{1/4 - N^2}{z^2},
\]
where $J_N(x)$ denotes the Bessel function of the first kind.

Tracy and Widom \cite{TW2} constructed another remarkable commuting pair of integral and differential operators
associated to the Airy kernel
\[
\frac{A(z) A'(w) - A'(z) A(w)}{z-w}
\]
acting on $L^2(\tau, + \infty;dw)$, where $A(z)$ denotes the Airy function. They proved that 
this integral operator commutes with
\[
\partial_z (\tau-z) \partial_z - z(\tau-z).
\]
This commuting pair and a modification of the one for the Bessel kernel were used for the asymptotic analysis of the level spacing distribution functions 
of the edge scaling limits of the Gaussian Unitary Ensemble and the Laguerre and Jacobi Ensembles in \cite{TW1,TW2}.

Fourth and sixth order commuting differential operators were constructed in  \cite{Grunbaum1996} by using 
two 1-step Darboux transformations from the Bessel functions with parameters $N =3/2$, $5/2$.
Examples of differential operators commuting with the finite Laplace transform were also found in \cite{Bertero}.

All of the above commuting pairs (and others that have been found in the past) fit into one general scheme: 
commuting differential operators were constructed for an integral kernel of the form
\begin{equation}
\label{K-oper}
K_{\psi}(z,w) := \int_{\Gamma_1} \psi(x,z)\psi^\dag(x,w) dx
\end{equation}
acting on $L^2(\Gamma_2;dw)$, where $\Gamma_1$ and $\Gamma_2$ are contours in ${\mathbb{C}}$, $\psi(x,z)$ is a 
wave function for the KP hierarchy and $\psi^\dag(x,z)$ is its adjoint wave function defined below.
(For Slepian's Bessel-type example, we get a kernel of this form for the square of Slepian's operator $\mathcal{E}$ defined above.)
\subsection{The adelic Grassmannian}\mbox{}
\label{1.2}
The decades of the 1960s and 1970s witnessed a rich development in the area of ``integrable
systems with an infinite number of degrees of freedom". It was soon realized that the fact that
certain nonlinear differential equations could be explicitly solved was linked to several parts
of mathematics including algebraic geometry, representation theory of infinitely dimensional
groups, etc. By 1981 M. Sato and his school noticed that the solutions of the KP hierarchy can
be understood in terms of an infinite dimensional Grassmannian and that a central role is played
by a ``tau function" $\tau(t_1,t_2,\dots)$ of infinitely many variables. In the simplest case this is a character of a representation of $\text{GL}_N(\bbc)$, and more
generally it is an appropriate linear combination of these, with coefficients satisfying the Pl\"ucker quadratic conditions. These efforts were pushed further by several people,
and the reader can consult \cite{Date,Sato,SW,Wilson-93}.

Wilson's adelic Grassmannian $\Grad$, \cite{Wilson-93}, is 
the infinite dimensional Grassmannian consisting of all subspaces $W \subset \bbc(z)$ of the form
\begin{equation}
\label{W}
W = \frac{1}{q(z)} \{f(z)\in\bbc[z]: \langle \chi_j(z), f(z) \rangle  = 0, 1 \leq j \leq n \} \quad \mbox{for some} \quad n \geq 0,
\end{equation}
where $\chi_j(z)$ are distributions on $\bbc$ that are linear combinations of derivatives of delta functions 
such that the support of each $\chi_j(z)$ is a singleton (adelic condition). Here $q(z)$ is the monic polynomial of degree $n$ 
whose roots are the supports of $\chi_1(z), \ldots, \chi_n(z)$. Each $W \in \Grad$ gives rise to a {\em{wave function}} $\psi_W(x,z)$ of
the KP hierarchy \cite{SW,vanMoerbeke} of the form $\frac{h(x,z)}{p(x)q(z)}e^{xz}$
for some polynomials $p(x),q(z)$ and $h(x,z)$, see \S \ref{2.1}. 

In terms of the tau function $\tau_W(t_1,t_2,\dots)$ of $W\in\Grad$, 
the wave function and the adjoint wave function of $W$ have expressions given by the following modified version of Sato's formula and its dual \cite[Theorem 1.5]{vanMoerbeke}
\begin{equation}\label{Sato formula}
\psi_W(x,z) = \frac{\tau_W(\mathbf t-[1/z])}{\tau_W(\mathbf t)}e^{xz},\ \ \ \psi_W^\dag(x,z) = \frac{\tau_W(\mathbf t +[-1/z])}{\tau_W(\mathbf t)}e^{xz}.
\end{equation}
Here $\mathbf t = (x,t_2,t_3,\dots)$ for some parameters $t_i$ and $[z] := (z,z^2/2,z^3/3,\dots)$.
Furthermore, a modified version of the differential Fay identity \cite[Theorem 1.9]{vanMoerbeke},
\begin{equation}\label{Fay identity}
\psi_W(x,z)\psi_W^\dag(x,w) = \frac{1}{z+w}\frac{\partial}{\partial x}\left[e^{x(z+w)}\frac{\tau_W(\mathbf t-[1/z] + [-1/w])}{\tau_W(\mathbf t)}\right]
\end{equation}
allows us to explicitly integrate our expression for the kernel \eqref{K-oper}, as in \eqref{tau formula for the kernel}.
We note that we do not rely on the $\tau$ function, Sato's formula, or the differential Fay identity for any of the results in the paper.  However Sato's formula and the differential Fay identity is utilized for \emph{examples} in Section 6.
\begin{remark}
The adjoint wave function $\psi^\dag_W(x,z)$ which we use throughout the paper is the wave function of the point of the adelic Grassmaannian obtained from $W$ by the action of a certain adjoint involution. It is closely related to the dual wave function of $W$ \cite[Theorem 1.1(f)]{vanMoerbeke}, see Remark \ref{dual vs adjoint}. The use of the adjoint wave function allows us to treat the two factors of \eqref{K-oper} on an equal footing, which plays a key role in all proofs. 
\end{remark}

The adelic Grassmannian $\Grad$ plays a ubiquitous role as the classifying space of objects of very different nature. 
Firstly, Wilson proved in \cite{Wilson-93} that $\Grad$ classifies all (normalized) rank $1$ bispectral functions. 
The bispectral problem posed by Duistermaat and Gr\"unbaum in \cite{DG86}, asks for a classification of all 
meromorphic functions $\psi(x,z)$ on a subdomain $\Omega_1 \times \Omega_2 \subseteq \bbc^2$ 
which are eigenfunctions in both variables, i.e.
\[
L(x,\partial_x)\cdot\psi(x,z) = g(z)\psi(x,z) \quad \; \mbox{and} \quad \; 
M(z,\partial_z)\cdot\psi(x,z) = h(x)\psi(x,z)
\]
for a pair of differential operators $L(x,\partial_x)$, $M(z,\partial_z)$ on $\Omega_1$, $\Omega_2$ and nonconstant meromorphic functions $h(x)$, $g(z)$
on $\Omega_1$, $\Omega_2$. The rank of such a {\em{bispectral function}} $\psi(x,z)$ is defined to be the greatest common divisor of the orders of 
the operators $L(x,\partial_x)$ for which $\psi(x,z)$ is an eigenfunction.  Note that this definition of rank is consistent with the one in \cite{Wilson-93}.

The following is a partial list of other situations in which $\Grad$ arises as a classifying space: 
\begin{enumerate}[(i)]
\item The space ${\mathrm{Gr}}^{\mathrm{ad}}$ parametrizes the set of rational solutions of the KP hierarchy vanishing 
as $x \to \infty$, introduced and studied by Krichever \cite{Krichever}. Wilson \cite{Wilson-98} constructed a decomposition 
of the adelic Grassmannian 
\begin{equation}
\label{CM}
{\mathrm{Gr}}^{\mathrm{ad}} = \bigsqcup_{k \geq 0}
{\mathcal{CM}}_k
\end{equation}
and proved that each stratum ${\mathcal{CM}}_k$ is isomorphic to a completion of the complexification of the real
phase space of the Calogero--Moser system on $k$ particles, constructed via Hamiltonian reduction by 
Kazhdan--Kostant--Sternberg \cite{KKS}.

\item Etingof and Ginzburg \cite{EG} proved that the Calogero--Moser stratum ${\mathcal{CM}}_k$ parametrizes the simple modules of the 
rational Cherednik algebra associated to the symmetric group $S_k$ when the deformation parameter $t=0$. 
\item Cunnings--Holland \cite{CH} and Berest--Wilson \cite{BW00} proved that $\Grad$ classifies the equivalence classes of one-sided ideals of the first Weyl algebra. 
\item Berest and Chalykh \cite{BC} proved that $\Grad$ is a classifying space for the isomorphism classes of strict $A_\infty$-modules over the 
first Weyl algebra.
\end{enumerate}
The adelic Grassmannian is considered to be a quantum analog of the Hilbert scheme points in $\bba^2$
$$\bigsqcup_{k \geq 0} {\mathrm{Hilb}}_k(\bba^2),$$
which also plays a key role as a universal classifying space \cite{Nakajima}.

The paper is organized as follows. In section \ref{results} we state our main results: Theorems A-E. Section \ref{sec:Wilson} contains background material 
on Wilson's adelic Grassmannian and algebro-geometric stuctures that we use in later sections. Theorem E is proved in Section \ref{sec:Thm E} by an algebro-geometric 
approach to the structure of Fourier algebras. In Section \ref{sec:Thm A-D} we prove Theorems A-D by using methods from integrable systems and noncommutative algebra. 
Finally, Section \ref{sect:examples} contains applications of our results to different examples.
\section{Statement of Results}\label{results}
\subsection{A general theorem on reflective integral operators}
\label{1.3}
The unifying feature of the diverse lines of research in \S \ref{1.1} is that commuting differential operators were constructed 
for the integral kernels given by \eqref{K-oper} for very concrete choices of bispectral functions $\psi(x,z)$. Previous results were based on direct 
computations done on a case-by-case basis. They did not offer insight into the nature of the commuting 
differential operator and essentially succeeded because a commuting operator of low order $(\leq 6)$ could be constructed for each concrete situation.
Consequently, the construction of commuting differential operators for the kernels associated to the 
points of the second Calogero--Moser stratum ${\mathcal{CM}}_2$ remained wide open despite the fact that those wave functions were 
prominently featured in many settings: they come from the simplest rational solutions of the KdV equation by Airault, McKean, and Moser \cite{AMM77}
and the simplest nontrivial bispectral functions of Duistermaat and Gr\"unbaum \cite{DG86}.

In this paper we give a general solution of these problems that is applicable to the integral operators associated to all points of $\Grad$.
It is based on a conceptual way of constructing the commuting differential operators from Fourier algebras associated to bispectral functions
and proving sharp bounds on their growth using algebro-geometric methods. 

One of our key ideas is that the intrinsic properties of the adelic Grassmannian naturally lead us to a more general property of its integral operators
than a plain commutativity:
\begin{definition}
An integral operator $I$, acting on $L^2(\Gamma)$ for a contour $\Gamma \subset {\mathbb{C}}$ and taking a function $f(w)$ to a new function $I(f)(z)$, 
is said to \vocab{reflect} a differential operator $R(z,\partial_z)$ with rational coefficients if 
$$R(-z,-\partial_z)\cdot I(f)(z) = I(R(w,\partial_w)\cdot f)(z)$$
on a dense subspace of $L^2(\Gamma)$.
\end{definition}

This condition can be restated in several equivalent forms. For instance, it is equivalent to requiring that $I \circ \sigma : L^2(-\Gamma) \to L^2(\Gamma)$ commutes with 
$R(z,\partial_z)$ on a dense  subspace of $L^2(-\Gamma)$ where $\sigma : L^2(-\Gamma) \to L^2(\Gamma)$ is the operator $\sigma ( f(x)) = f(-x)$. 

To state our results we need to introduce some background on bispectral functions and the Grassmannian $\Grad$. For a bispectral function $\psi(x,z)$, 
define the \vocab{left and right Fourier algebras} of $\psi$ respectively by
\begin{align}
\label{Fx}
\mathcal F_x(\psi) &:= \{S(x,\partial_x):\ \exists R(z,\partial_z)\ \text{such that}\ S(x,\partial_x)\cdot\psi(x,z) = R(z,\partial_z)\cdot\psi(x,z)\}, \\
\label{Fz}
\mathcal F_z(\psi) &:= \{R(z,\partial_z):\ \exists S(x,\partial_x)\ \text{such that}\ S(x,\partial_x)\cdot\psi(x,z) = R(z,\partial_z)\cdot\psi(x,z)\}.
\end{align}
If $\psi$ is nontrivial (i.e. $\psi(x,z) \neq \lambda(x) \mu(z)$ for some meromorphic functions
$\lambda(x)$ and $\mu(z)$ on $\Omega_1$ and $\Omega_2$), then there is a canonical antiisomorphism
\begin{equation}
\label{bpsi}
b_\psi : \mathcal{F}_x(\psi) \to \mathcal{F}_z(\psi)
\quad \mbox{defined by} \quad b_\psi(S(x,\partial_x)) := R(z,\partial_z),
\end{equation}
where $S(x,\partial_x)\cdot\psi(x,z) = R(z,\partial_z)\cdot\psi(x,z)$.
When $\psi(x,z)$ is appropriately normalized, the operators in $\mathcal{F}_x(\psi)$ and  $\mathcal{F}_z(\psi)$ have rational coefficients. 

Wilson \cite{Wilson-93} introduced three involutions $a,b,s$ of $\Grad$. The {\em{adjoint involution}} $a$ acts on a space $W \in \Grad$ by taking 
its orthogonal complement for a certain inner product. For brevity we write, with more details coming in Sect. \ref{2.1},
\[
\psi^\dag_W(x,z) = \psi_{a(W)}(x,z).
\]
\begin{remark}\label{dual vs adjoint}
The adjoint wave function $\psi^\dag_W(x,z)$ is related to the dual wave function $\psi^*_W(x,z)$ of \cite{Date} by $\psi^\dag_W(x,z) = \psi^*_W(x,-z)$.
\end{remark}

The {\em{bispectral involution}} $b$ interchanges the roles of $x$ and $z$ in a wave function $\psi_{b(W)}(x,z) = \psi_W(z,x)$. The {\em{sign involution}} $s$ is given by
$\psi_{s(W)}(x,z) = \psi_W(-x,-z)$. We will also need a fourth one, the {\em{Schwarz involution}} $c$, given by $\psi_{c(W)} (x,z)=\ol{\psi_W(\ol x,\ol z)}$. 

Define the {\em{degree}} of $W \in \Grad$ by setting $\deg W =n$ for the minimal $n$ for which $W$ can be represented in the form \eqref{W}. 
Set
\[
\ub(W) = 2 \left\lceil \frac{ \deg(W) + \deg(a W)}{2} \right\rceil \cdot
\]
\medskip

\noindent
{\bf{Theorem A.}} 
{\em{Consider $W \in \Grad$ and smooth contours $\Gamma_1, \Gamma_2 \subset \bbc$ starting from some $r,t \in \bbc$ and going to $+\infty$ in the strip $|\Im \, z|<  \const$, such that 
$\Gamma_1 \times \Gamma_2$ avoids the poles of $\psi_W(x,-z)$ and $\psi_W^\dag(x,-z)$. The integral operator $T_W$ on $L^2(\Gamma_2)$ with kernel 
$$
K_W (z,\xi) := \int_{\Gamma_1} \psi_W(x,-z)\psi_W^\dag(x,-\xi) dx
$$
reflects a non-constant differential operator $R_{r,t}(z,\partial_z) \in \mathcal F_z(\psi)$
of order $\leq \min(\ub(W), \ub(bW))$ with the properties that 
\begin{enumerate}
\item[{\em{(a)}}] $R_{r,t}(z,\partial_z)$ is independent of the choice of paths $\Gamma_1, \Gamma_2$ from $r$ and $t$, resp. and
\item[{\em{(b)}}] the coefficients of $R_{r,t}(z,\partial_z)$ are rational functions in $r,t$ and $z$.
\end{enumerate}
}}

The set of polynomial $\tau$-functions of the KP heirarchy lies inside $\Grad$.
The integral operators built from those $\tau$-functions will in general reflect a differential operator, but may not commute with a differential operator, as illustrated in Sect. \ref{sect:examples}.

The simplest example of Theorem A arises when $\psi_W(x,z) = \psi_W^\dag(x,z) = e^{xz}$.  In this case, the theorem tells us
that
$$T_W(f)(z) = \int_t^\infty \frac{1}{z+w}e^{-r(z+w)}f(w)dw\ \ \text{reflects}\ \ R(z,\partial_z) = (z-t)\partial_z-rz.$$
Details are in Sect. \ref{sect:examples}.
\begin{remark}
If $r$ is the unique finite endpoint of $\Gamma_1$, then $K_W(z,\xi)$ equals the specialization of
\begin{equation}\label{tau formula for the kernel}
K_{W}(z,\xi;\mathbf t) = \frac{1}{z+\xi}\left[e^{-r(z+\xi)}\frac{\tau_W(\mathbf t-[-1/z] + [1/\xi])}{\tau_W(\mathbf t)}\right]
\end{equation}
at $t_i=0$ for all $i>1$ where here $\mathbf t = (r,t_2,t_3,\dots)$ and $\tau_W(t_1,t_2,\dots)$ is the $\tau$-function of the point $W\in\Grad$.
Note that we specialize in order to associate each point $W\in\Grad$ to a \emph{unique} kernel $K_W(z,\xi;(r,0,\dots,0))$.
The family of kernels $K_{W}(z,\xi;\mathbf t)$ for more general values of $\mathbf t$ define the evolution of the kernel under the KP flows.
We will not use this explicit expression for the kernel until the examples in Section 6.
\end{remark}
%
\subsection{The Fourier and Laplace pictures}
\label{1.4}
Returning to the setting of time-band limiting, 
for $W \in \Grad$ we can relate the integral operators from Theorem A and those of the form $EE^*$ 
by considering the following generalized Laplace transform and its adjoint with respect to the standard inner product:
\[
L_W : f(x)\mapsto \int_0^\infty \psi_W(x,-z)f(x)dx, \quad L^*_W: f(z)\mapsto \int_0^\infty \ol{\psi_W(x,-z)}f(z)dz.
\]
(In the special case $W=\bbc[z]$, i.e., $\psi_W(x,z) = e^{xz}$,
the operator $L_W$ is precisely the Laplace transform.)
The \vocab{time and band-limited} versions of these are 
$$({\mathcal{E}}_W f)(z)   = (\chi_{[t,\infty)}L_W \chi_{[r,\infty)} f)(z)  = \int_r^\infty \psi_W(x,-z) f(x) dx$$
for $z\geq t$, and 
$$({\mathcal{E}}^*_W f)(x) = (\chi_{[r,\infty)}L^*_W \chi_{[t,\infty)} f)(x) = \int_t^\infty \ol{\psi_W(x,-z)}f(z)dz$$
for $x\geq r$, respectively.
The singular values of ${\mathcal{E}}_W$ are the eigenvalues of the integral operator on $L^2(t, \infty)$
\begin{equation}\label{tblimit laplace equation}
({\mathcal{E}}_W {\mathcal{E}}^*_W f)(z) = \int_t^\infty K_W (z,\xi)f(\xi)d\xi, \quad \mbox{where} \quad
K_W (z,\xi) = \int_r^\infty \psi_W(x,-z)\ol{\psi_W(x,-\xi)}dx.
\end{equation}
Applying Theorem A to this situation, leads us to the following result.
\medskip

\noindent
{\bf{Theorem B.}} {\em{Let $W \in \Grad$ be fixed under the involution $ac$, 
and $r, t \in \bbc$ be such that $[r,+ \infty) \times [t, + \infty)$ avoids the poles of $\psi_W(x,-z)$. Then 
the integral operator ${\mathcal{E}}_W {\mathcal{E}}^*_W$ reflects the non-constant 
differential operator $R_{r,t}(z,\partial_z)\in\mathcal F_z(\psi)$ from Theorem A.
}}
\medskip

By performing a \vocab{90 degree rotation} in the complex variable $z$, we move from the Laplace transform picture to the Fourier transform picture.
We prove in Sect. \ref{sec: moving to commutative} that in this way one can convert the reflected differential operators from Theorem A to {\em{commuting pairs}} in the Fourier picture.
For $W \in \Grad$, replace $L_W$ and $L_W^*$ with the operators on $L^2(\bbr)$
\[
F_W : f(x)\mapsto \int_{-\infty}^\infty \psi_W(x,iz)f(x)dx, \quad
F_W^*: f(z)\mapsto \int_{-\infty}^\infty \ol{\psi_W(x,iz)}f(z)dz.
\]
(In the special case $W=\bbc[z]$, $F_W$ is the Fourier transform.)
Consider the time and band-limited operators 
\begin{align*}
&(E_W f)(z)  = (\chi_{[t,\infty)} F_W \chi_{[r,\infty)} f)(z)  = \int_r^\infty \psi_W(x,iz) f(x) dx, \\
&(E^*_W f)(x) = (\chi_{[r,\infty)}F^*_W \chi_{[t,\infty)} f)(x) = \int_t^\infty \ol{\psi_W(x,iz)}f(z)dz
\end{align*}
and the self-adjoint operator on $L^2(t, \infty)$
\begin{equation}\label{tblimit fourier equation}
(E_W E^*_W f)(z) = \int_r^\infty \int_t^\infty \psi_W(x,i z)\ol{\psi_W(x,i\xi)}f(\xi)d \xi dx.
\end{equation}

\noindent
{\bf{Theorem C.}} {\em{
Let $W\in\Grad$ and let $R_{r,t}(z,\partial_z)$ be the operator from Theorem A.
\begin{enumerate}[(i)]
\item
Let $\Gamma_1$ and $\Gamma_2$ be two smooth contours starting from $r,t\in\bbc$ respectively and going to $+\infty$ in a strip with bounded real component, such that $\Gamma_1\times\Gamma_2$ avoids the poles of $\psi_W(x,iz)$ and $\psi_W^\dag(x, -i\xi)$.
Then the integral operator on $L^2(\Gamma_2)$ defined by
$$U_W(f)(z) = \int_{\Gamma_1}\int_{\Gamma_2} \psi_W(x,iz)\psi_W^\dag(x,-i\xi)f(\xi) d\xi dx$$
commutes with the 90 degree rotation $\wt R_{r,t}(z,\partial_z)$ of the differential operator $R_{t,z}(z,\partial_z)$, defined by
\begin{equation}\label{90 degree rotation}
\wt R_{r,t}(z,\partial_z) = R_{r,it}(iz,-i\partial_z).
\end{equation}
\item
Assume additionally that $W$ is fixed under the involution $ac$ and let 
and $r, t \in \bbr$ be such that $[r,+ \infty) \times [t, + \infty)$ avoids the poles of $\psi_W(x,iz)$. 
Then the integral operator $E_W E_W^*$ commutes with the nonconstant differential operator $R_{r,it}(iz, -i \partial_z)$.
\end{enumerate}
}}
\medskip

For analytic reasons, one cannot deduce Theorem C from Theorem B by an elementary change of variables.
\subsection{Simultaneous reflectivity}
\label{1.5}
The proofs of Theorems A-C produce not just a single reflected/commuting operator, but a large algebra of such operators 
(inside the Fourier algebra  ${\mathcal{F}}_z(\psi_W)$). We use this to prove the existence of \vocab{universal operators} 
which are simultaneously reflected by  (or commute with) finite collections of integral operators.
\medskip

\noindent
{\bf{Theorem D.}} (i) {\em{Consider any finite collection $W_1, \ldots, W_N \in \Grad$ and let $T_1, \ldots, T_N$ be the associated integral operators as 
in Theorem A for the same values of $r$ and $t$. There exists a non-constant differential operator in $\bigcap_k\mathcal F_z(\psi_{W_k})$ 
simultaneously reflected by each of the operators $T_k$.}}

(ii) {\em{If, in addition, $W_1, \ldots W_N \in \Grad$ are fixed under the involution $ac$ as in Theorems B and C, then there exists 
a differential operator $R_{r,t}^{\text{univ}} (z, \partial_z)$ which is simultaneously reflected by all integral operators ${\mathcal{E}}_{W_j} {\mathcal{E}}_{W_k}^*$
for $1 \leq j, k \leq N$. This differential operator has rational coefficients in $z,r,t$ and $\wt R_{r,t}(z,\partial_z) := R_{r, it}^{\text{univ}}(iz,-i \partial_z)$ commutes with all
integral operators $E_{W_j} E_{W_k}^*$ for $ 1 \leq j, k \leq N$.}}
\medskip

The integrals considered by Landau, Pollak, Slepian and Mehta \cite{Mehta,S,SP} involve finite intervals, while those considered by Tracy and Widom \cite{TW1,TW2} involve semi-infinite intervals.
Our situation is similar to the latter.
Commutativity results for finite intervals concerning subfamilies of wave functions from the adelic grassmannian were obtained in \cite{CY18,Grunbaum1996}.
Such results do not extend in the generality considered in this paper.

Since the Fourier algebra of $\exp(xz)$ is the first Weyl algebra, we can force all coefficients of 
the differential operators in both parts of the theorem to be polynomials in $z$.

Theorems A--D open up avenues of broad new applications of integrable systems to the spectral analysis and numerical properties of integral operators, going
far beyond sinc, Bessel and Airy kernels. We will return to this in future publications.

\begin{remark}
All commutativity and reflect commutativity identities in this paper are proved on dense subspaces of Hilbert spaces. We do not
consider the finer analytic problem of determining maximal subspaces on which these commutativity properties hold. Even in the situation of
the sinc kernel this has only recently been addressed in \cite{K}.
\end{remark}
\subsection{The size of bifiltrations of Fourier algebras} 
\label{1.6}
A key ingredient in the proofs of Theorems A-D are the following sharp estimates of the growth rate 
of Fourier algebras, which are of independent interest.
The left Fourier algebra $\mathcal F_x(\psi)$ of a bispectral function $\psi(x,z)$ has a natural $\bbn \times \bbn$-filtration
\begin{equation}
\label{Fx-filt}
\mathcal F_x^{\ell,m}(\psi) := \{L(x,\partial_x)\in\mathcal F_x(\psi):\ord \, L(x,\partial_x)  \leq \ell,\ \ord \, b_\psi(L(z,\partial_z)) \leq m\}, \quad
\ell, m \in \bbn
\end{equation}
and similarly does the right one $\mathcal F_z(\psi)$, see \S \ref{3.1}. (Here and below $\bbn = \{0, 1, \ldots \}$.)
We obtain an exact formula for their sizes:
\medskip

\noindent
{\bf{Theorem E.}} {\em{Let $W \in \mathcal{CM}_k$ and recall \eqref{CM}.}}

(i) {\em{For $m\geq \deg(W) + \deg(aW)-1$ and $\ell\geq 0$, as well as for $m\geq 0$ and $\ell\geq \deg(bW) + \deg(abW)-1$, we have
$$\dim\mathcal F^{\ell,m}_z(\psi_W) \geq (\ell+1)(m+1)- 2 k,$$
with equality when both $m\geq \deg(W) + \deg(aW)$ and $\ell\geq \deg(bW) + \deg(abW)$.}}

(ii) $k \leq \min \big( \deg(W)\deg(a W), \deg (bW) \deg (ab(W) \big)$. 
\subsection{Strategy of the proofs}
We prove that, for $W \in \Grad$, the left Fourier algebra $\mathcal F_z(\psi_W)$ has two more incarnations
(see Theorem \ref{ident} and Remark \ref{rmk-incarnations}).
It is isomorphic to
\begin{enumerate}
\item[(a)] the algebra of global differential operators on a line bundle of the affine spectral curve of the solution of the KP hierarchy corresponding 
to $W \in \Grad$;
\item[(b)] the subalgebra of the ``additional symmetries" of the KP hierarchy see Remark \ref{rmk-incarnations} that fix the point $W \in \Grad$. 
\end{enumerate}

Theorem E is proved by using algebro-geometric methods and the incarnation (a) of $\mathcal F_z(\psi)$. 
Theorems A-D are then deduced from it by using a combination of analytic methods and methods from integrable systems. 
The joint domain of commutativity/reflectivity of integral and differential operators is the Sobolev space 
$W^{\ell,2}$. (Note that if $f\in W^{n+1,2}$, then its first $n$ derivatives must vanish at infinity.)
By this, we mean that the integral and differential operators commute on a dense subset of $W^{\ell,2}$.
The intrinsic reason for the appearance of reflectivity, as opposed to plain commutativity, is
that the product $ab$ of the adjoint and bispectral involutions of $\Grad$ 
is not an involution of $\Grad$. It is a fourth order automorphism of $\Grad$ such that that $(ab)^2 = s$. 
This key property was used by Wilson \cite{Wilson-93} in 1993; here we observe that it is at the heart of the prolate-spheroidal phenomenon of Landau, Pollak and Slepian. 

The $W$-constraints of a tau function of the KP hierarchy are precisely the elements of the algebra in the incarnation (b) above. 
In this way Theorem E determines the exact size of the $W$-constraints of the tau function of every point of the adelic Grassmannian. 
The $W$-constraints of tau functions of the 
Toda and multi-component KP hierarchies play a major role in many settings: the string equation for the Airy tau function arising 
as a partition function in quantum gravity and the generating function for the intersection theory on moduli spaces of curves \cite{AvM92,Kon,Wit}, 
the $W$-constraints of Gromov--Witten invariants \cite{OP} and the total descendant potential of a simple singularity \cite{BM}, and
in deriving a system of partial differential equations for the distributions of the spectra of coupled random matrices \cite{AvM99}. 

In \cite{Iliev} Iliev constructed quantum versions of the Calogero--Moser spaces. In a forthcoming publication we will prove that they give rise to difference operators 
that satisfy appropriate versions of Theorems A-D.

Some of the results in this paper were announced in \cite{CGYZ}. 

\section{Wilson's adelic Grassmannian and algebro-geometric structures}\label{sec:Wilson}
This section contains background material on Wilson's adelic Grassmannian, related algebro-geometric stuctures and  
bispectrality that will be used in later sections.
\subsection{Wilson's adelic Grassmannian}
\label{2.1}
Let $a\in\bbc$ and let $\mathscr C_a$ denote the $\bbc$-linear span of all discrete distributions on $\bbc$ of the form $\delta^{(k)}(z-a)$ for $k\geq 0$ 
an integer, and define $\mathscr C := \bigoplus_{a\in\bbc} \mathscr C_a$.
An element $\chi(z)\in\mathscr C$ is called \vocab{homogeneous} if $\chi(z)\in \mathscr C_a$ for some $a\in\bbc$.
A subspace $C\subseteq\mathscr C$ is called \vocab{homogeneous} if it has a basis consisting of homogeneous elements, or equivalently if 
$C = \bigoplus_{a\in\bbc} C\cap \mathscr C_a$.
For any $C\in\mathscr C$ we define a subspace $V_C\subseteq\bbc[z]$ by
$$V_C := \{f(z)\in\bbc[z]: \langle \chi(z), f(z) \rangle  = 0\ \forall \chi\in C\}.$$
The \vocab{adelic Grassmannian} is the set
$$\Grad := \Big\lbrace q(z)^{-1} V_C\subseteq \bbc(z) : C \subset \mathscr C \; \; \mbox{homog.},\ \dim(C) <\infty,\ q(z) = \prod_{a\in\bbc}(z-a)^{\dim(C\cap \mathscr C_a)}\Big\rbrace.$$
Note that the subspace  $C \subset \mathscr C$ is not uniquely determined by the plane $W = q(z)^{-1} V_C$. 
The sign and adjoint involutions of $\Grad$ are given by
\begin{equation}
\label{aW}
s(W) := \{f(-z): f(z)\in W\}, \quad a(W) := \{f(z) : \oint_{|z|=1} f(z)g(-z) \frac{dz}{2\pi i} = 0\ \forall g(z)\in W \},
\end{equation}
see \cite[Sect. 8]{Wilson-93}. We refer the reader to \cite{SW} and \cite{vanMoerbeke} for background on the KP hierarchy in the frameworks
of the Segal--Wilson and Sato Grassmannians, respectively. (The Wilson adelic Grassmannian can be treated in either framework.)
The wave function of the KP hierarchy \cite{Wilson-98} of a plane $W = q(z)^{-1} V_C \in \Grad$ is given by 
\begin{equation}
\label{psiW}
\psi_W(x,z) := \frac{1}{p(x)q(z)}P(x,\partial_x)\cdot e^{xz} = \frac{1}{p(x)q(z)}Q(z,\partial_z)\cdot e^{xz}
\end{equation}
with $p(x) \in\bbc[x]$, $P(x,\partial_x) \in \bbc[x,\partial_x]$ and $Q(z,\partial_z) \in \bbc[z,\partial_z]$
such that $p(x)^{-1} P(x, \partial_x)$ is the unique monic differential operator of order $\dim C$ with kernel
$\{ \langle \chi(z), e^{xz} \rangle : \chi(z) \in C \}$. The differential operator $Q(z, \partial_z)$ is uniquely 
determined from the second equality in \eqref{psiW}. 
The above differential operators and polynomials satisfy 
\begin{equation}
\label{psiW2}
\wt P(x,\partial_x)\frac{1}{\wt p(x)p(x)}P(x,\partial_x) = q(\partial_z)\wt q(\partial_z)\ \ \text{and}\ \ \wt Q(z,\partial_z)\frac{1}{\wt q(z)q(z)}Q(z,\partial_z) = p(\partial_x)\wt p(\partial_x).
\end{equation}
for some 
$\wt{p}(x) \in\bbc[x]$, $\wt{q}(z) \in\bbc[z]$, $\wt{P}(x,\partial_x) \in \bbc[x,\partial_x]$ and $\wt{Q}(z,\partial_z) \in \bbc[z,\partial_z]$.
Wilson \cite{Wilson-93} proved that the meromorphic functions $\psi_W(x,z)$ on $\bbc^2$ (for $W \in \Grad$) 
exhaust all normalized bispectral functions of rank 1. 

Each function $\psi(x,z)$ with the above properties \eqref{psiW}-\eqref{psiW2} arises as the wave function of some $W_\psi \in \Grad$, \cite{BHY97}, 
which is reconstructed from $\psi(x,z)$ as follows. 
We have $\psi(x,z) = e^{xz}\frac{h(x,z)}{p(x)q(z)}$ for some polynomial $h(x,z)$. Define a finite dimensional homogeneous subspace 
of $\mathscr C$ by 
\begin{equation}
\label{Cpsi}
C := \{\chi\in\mathscr C: \langle \chi (z), e^{xz}h(x,z) \rangle = 0\}.
\end{equation}
Then $W_\psi = \frac{1}{q(z)}W_{C}$. Furthermore, 
\begin{equation}
\label{Wpsi}
W_\psi = \vspan_{\bbc}\left\lbrace \partial_x^n(p(x)\psi(x,z))|_{x=0}\right\rbrace
= q(z)^{-1} \big( \bbc[z]\cdot P(x, \partial_x) \big),
\end{equation}
where $\bbc[z]$ is endowed with the right $\bbc[x,\partial_x]$-module 
structure by identifying $z^k$ with the equivalence class $[\partial_x^k]$ in $\bbc[x,\partial_x]/(x\bbc[x,\partial_x])$.
Eq. \eqref{Wpsi} implies that the Schwarz involution of $\Grad$ from the introduction is given by
\[
c(W) := \{\ol{f(\ol{z})}: f(z)\in W\}.
\]

The action of the sign, bispectral and Schwarz involutions on wave functions was described in \S \ref{1.3}.
The adjoint involution acts on wave functions by
\[
\psi_{a(W)}(x,z) = \frac{1}{\wt{p}(x)\wt{q}(z)}\wt{P}(x,\partial_x)^\dag \cdot \psi_{\exp}(x,z)
\]
in the notation \eqref{psiW}-\eqref{psiW2} (see \cite[Prop. 7.3]{Wilson-93} and \cite[Prop. 1.7]{BHY97}), where 
$(\cdot)^\dag$ denotes the formal adjoint of a differential operator, constructed without taking complex conjugates of its coefficients
\begin{equation}\label{formal adjoint formula}
\left(\sum_{j=0}^n a_j(x)\partial_x^j\right)^\dag = \sum_{j=0}^n (-1)^j\partial_x^ja_j(x).
\end{equation}
\begin{remark}
Note that this differs from the adjoint operator 
\begin{equation}\label{adjoint formula}
\left(\sum_{j=0}^n a_j(x)\partial_x^j\right)^* = \sum_{j=0}^n (-1)^j\partial_x^j \overline{a_j(x)}
\end{equation}
associated to the standard hermitian inner product. We only use the latter in Sect. \ref{sect:examples}.
\end{remark}

It follows from
these expressions and from \eqref{Cpsi} that
\begin{equation}
\label{equalities}
\deg (W) \leq \ord P, \; \; \deg(a (W)) \leq \ord \wt{P}, \; \; 
\deg( b (W)) \leq \ord Q, \; \; \deg (ab (W)) \leq \ord \wt{Q}.
\end{equation}
It is easy to show that, for every $W \in \Grad$, there exists a presentation of $\psi_W$ of the form 
\eqref{psiW}--\eqref{psiW2} such that all of the above inequalities become equalities.

The base point of the adelic Grassmannian is $W_0 = \bbc[z]$. It is fixed by the involutions $a,b,s,c$ and ${\mathcal{CM}}_0 = \{W_0 \}$. 
The wave function of $W_0$ is the exponential function
\[
\psi_{\exp}(x,z) = e^{xz}.
\]

\begin{example}\label{ex2.1}
Let $n \in \bbn$ and consider the bispectral meromorphic function
\[
\psi_n(x,z) := e^{xz}\sum_{j=0}^n\frac{n!}{j!}\frac{(-1)^{n-j}}{(xz)^{n-j}} \cdot 
\]
We have 
\[
\psi_{n+1}(x,z) = z^{-1} x^{-1} \left(x\partial_x - 1 \right)\psi_n(x,z),
\]
and thus, 
\[
\psi_n(x,z) = x^{-n} z^{-n} P(x, \partial_x) \psi_{\exp}(x,z) \quad \mbox{for} \quad P(x, \partial_x) =  (x\partial_x -1)(x\partial_x-2)\dots(x\partial_x-n).
\]
We determine $W_{\psi_n}$ using \eqref{Cpsi}: 
\begin{align*}
C
  & = \{\chi\in\mathscr C: \langle \chi(z), e^{xz}(xz-1) \rangle = 0\}\\
  & = \{\chi\in\mathscr C: \langle \chi(z), e^{xz} \rangle\in\ker P(x, \partial_x)\} = \vspan_\bbc\{\delta'(z),\delta''(z),\dots,\delta^{(n)}(z)\}.
\end{align*}
Therefore,
$$W_{\psi_n} = W_C = z^{-n}\{p(z)\in\bbc[z]: p^{(j)}(0) = 0,\ j=1,\dots,n\} = \vspan_\bbc\{1/z^n,z,z^2,z^3,\dots\}.$$
Alternatively, $W_{\psi_n}$ can be determined using the formulas in \eqref{Wpsi}. We leave the details to the reader.  
For brevity, set $W_n := W_{\psi_n}$. 
\qed
\end{example}

\begin{example}
\label{ex2.2}
Consider the Bessel bispectral functions defined for $\nu\in\bbc$ by
\begin{equation}\label{bispectral bessel}
\psi_{\Be(\nu)} (x,z) := (2 xz/\pi)^{1/2} K_{\nu +1} (xz)
\end{equation}
where $K_\nu(z)$ denote the modified Bessel functions.
They satisfy
\[
L_\nu(x, \partial_x) \psi_{\Be(\nu)} (x,z)  = z^2 \psi_{\Be(\nu)} (x,z), \quad 
x^2 \psi_{\Be(\nu)} (x,z) = L_\nu(z, \partial_z) \psi_{\Be(\nu)} (x,z)
\]
with respect to the Bessel differential operators
\[
L_\nu(x, \partial_x) := \partial_x^2 - \frac{\nu (\nu+1)}{x^2} \cdot
\]
The Bessel bispectral functions are Darboux transformations from each other \cite{DG86,BHY97} as follows:
\begin{align*}
\psi_{\Be(\nu + 1)}(x,z) &= \frac{1}{z} \left( \partial_x - \frac{\nu +1}{x} \right) \psi_{\Be(\nu)}(x,z), \\
\psi_{\Be(\nu-1)}(x,z) &= \frac{1}{z} \left( \partial_x +\frac{\nu}{x} \right) \psi_{\Be(\nu)}(x,z). 
\end{align*}
The function $\psi_{\Be(\nu)} (x,z)$ is invariant under the transformation $\nu \mapsto - \nu -1$. It has rank 1 for $\nu \in \bbz$ and rank 2 otherwise. 
For $n \in \bbn$, $\psi_{\Be(n)}(x,z) = x^{-n} z^{-n} P(x, \partial_x) \psi_{\exp}(x,z)$, where 
\[
P(x, \partial_x) = g( x \partial_x) \quad \mbox{and} \quad
g(x) = (x -2n+1)(x -2n+3)\dots(x-1).
\]
For $m>0$
$$z^m\cdot P(x, \partial_x) = [\partial_x^mg(x\partial_x)] = [g(x\partial_x+m)\partial_x^m] = (m-1)(m-3)\dots(m-2n+1)[z^m],$$
which is zero for $m=1,3,\dots,2n+1$ and a scalar multiple of $z^m$ otherwise. We determine $W_{\psi_{\Be(n)}}$ using the second fomula 
in \eqref{Wpsi}:  
$$W_{\psi_{\Be(n)}} = \frac{1}{z^n}\{f(z)\in\bbc[z]: [f(\partial_z)]\in \bbc[z]\cdot P(x, \partial_x)\} = \vspan_\bbc\{z^{j-n}: j\neq 1,3,\dots,2n-1\}.$$
Note that $W_{\psi_{\Be(n)}} = W_C$ for $C = \vspan_{\bbc}\{\delta^{(2j-1)}(z)\}_{j=1}^n$. 
Denote $W_{\Be(n)} := W_{\psi_{\Be(n)}}$. 
\qed
\end{example}

\subsection{$\Grad$ and line bundles on framed curves}
\label{2.2}
In this subsection we review background material of algebro-geometric nature which will play a major role in our proof of 
Theorem E in the next section.

Following Wilson \cite{Wilson-93}, we associate to each element $W \in \Grad$, a framed affine curve $X$, a line bundle $\mathcal L$ on $X$, and a generic local trivialization 
on $\mathcal L$. 
\begin{definition}[\cite{Wilson-93}] A {\em{framed affine curve}} is a pair $(X,\iota)$ consisting of an affine curve $X$ 
and a bijective map of schemes $\iota: \bba^1_\bbc\rightarrow X$.
\end{definition}

We will use the following slightly less common terminology. Following Segal--Wilson \cite{SW}, by a {\em{line bundle}} $\mathcal L$ on an irreducible affine curve $X$ we mean 
a sheaf $\mathcal L$ on $X$ which is torsion free and generically rank $1$.

Fix $W\in\Grad$ and consider the algebra
$$A_W := \{f(z)\in\bbc((z^{-1})): f(z)W\subseteq W\}.$$
One may easily verify from the definition of $\Grad$ that if $f(z)$ is a Laurent series in $z^{-1}$ and $f(z)W\subseteq W$ then $f(z)\in\bbc[z]$.
Thus $A_W\subseteq\bbc[z]$, inducing a natural morphism $\iota: \bba^1_\bbc\rightarrow  X = \spec(A_W)$.
Moreover, it is clear from the definition that if $C\in\mathscr C$ and $W = \frac{1}{q(z)}V_C$ then $q(z)^\ell\bbc[z]\subseteq A_W$ for some integer $\ell\gg0$.
Therefore the fraction field of $A_W$ is precisely $\bbc(z)$, so the morphism $\iota: \bba^1\rightarrow X$ is dominant.
By a result of Wilson \cite{Wilson-98}, the map $\iota$ is a bijection on points.

The vector space $W$ is an $A_W$-module which is torsion-free by virtue of the fact that $\bbc(z)$ is a domain.
Additionally, $W$ is finitely generated as a module by a degree argument combined with the above observation that $q(z)^\ell \bbc[z]\subseteq A_W$.
The inclusion $W\subseteq\bbc(z)$ induces an isomorphism $W\otimes_{A_W} K(A_W) = W\bbc(z) = \bbc(z)$, in particular making $W$ generically free of rank $1$.
Thus the usual functor from $A_W$-modules to sheaves on $X$ takes $W$ to a sheaf $\mathcal L = \wt W$ on $X$ which is torsion free and generically rank $1$.
The inclusion $W\subseteq\bbc(z)$ corresponds to a trivialization $\varphi: \mathcal L_\eta\rightarrow\mathcal O_{X,\eta}$, where here $\eta$ is the generic point of $X$.
To summarize, each point $W$ corresponds to
\begin{enumerate}
\item a framed affine curve $(X,\iota)$, 
\item a line bundle $\mathcal L$ on $X$, and
\item a local trivialization $\varphi: \mathcal L_\eta \rightarrow \mathcal O_{X,\eta}$.
\end{enumerate}

The degree valuation $\nu(f(z)) = -\deg(f(z))$ on $\bbc(z)$ defines a valuation $\nu_W$ on the fraction field $K(A_W)$ of $A_W$.
From the point of view of an abstract algebraic curve, this is a ``point at infinity" which can be used to complete $X$ to a projective curve $Y$.
Formally, this operation is completed by taking $Y$ to be $\Proj$ of the Rees ring defined by the filtration of $A_W$ induced by $\nu$.
Then $\spec(A_W)$ is isomorphic to $Y$ minus a point, denoted by $\infty$, and the valuation corresponds precisely to the divisor of $Y$ defined by $\infty$.

By \cite[Prop. 5.2]{BW04}, the map $W \mapsto (X, \mathcal{L})$ is a bijection from $\Grad$ to the set of isomorphism classes of 
pairs of a framed affine curve and line bundle on it.

\begin{example}
\label{ex2.4}
Let $n \in \bbn$. Consider the point $W_n= \vspan_\bbc\{1/z^n,z,z^2,z^3,\dots\} \in \Grad$ from Example \ref{ex2.1}.
The corresponding algebra $A_{W_n}$ is given by
$$A_{W_n} = \vspan_\bbc\{1,z^{n+1},z^{n+2},\dots\}.$$
Note in particular that $W_n = z^{-n}A_{W_n}$ so that $\mathcal L \cong \mathcal O_{X_n}$ for $X_n = \spec(A_{W_n})$.
The curve $X_n$ is one of the special extremal rational, unicursal curves considered in \cite{BW04,kouakou}.
In particular $X_n$ is the simplest unicursal curve having differential genus $n$, as described below.
\qed
\end{example}

\begin{example}
\label{ex2.5}
For $n \in \bbn$ consider the point 
$$W_{\Be(n)}= \frac{1}{z^n}\vspan_\bbc\{1,z^2,z^4,\dots, z^{2n},z^{2n+1},z^{2n+2},\dots\} \in \Grad$$
from Example \ref{ex2.2}. The corresponding algebra $A_{W_{\Be(n)}}$ is given by
$$A_{W_{\Be(n)}}=  \vspan_\bbc\{1,z^2,z^4,\dots, z^{2n},z^{2n+1},z^{2n+2},\dots\}.$$
Note again that $W_{\Be(n)} = z^{-n}A_{W_{\Be(n)}}$ so that $\mathcal L \cong \mathcal O_{X_{\Be(n)}}$ for $X_{\Be(n)} = \spec(A_{W_{\Be(n)}})$.
\qed
\end{example}

\subsection{Differential operators on line bundles}
Let $R$ be a commutative $\bbc$-algebra and let $M,N$ be $R$-modules. The set of $N$-valued differential operators on $M$ of order at most $m$ is the vector space
$$\mathcal{D}^m(M,N) := \{\varphi\in \Hom_\bbc(M,N):  \ad_r^{m+1}(\varphi) = 0\ \forall r\in R\}.$$
We write $\mathcal{D}(M,N) := \bigoplus_{m\geq 0} D^m(M,N)$, and for $M=N$, we write $\mathcal{D}^m(M)$ in place of $\mathcal{D}^m(M,M)$.
Note that $\mathcal{D}(M)$ inherits an algebra structure as a subalgebra of $\bbc$-linear endomorphisms of $M$. If $R$ is an integral domain 
with field of fractions $\mathbb{K}$, then 
\begin{equation}
\label{sub-diff}
\mathcal{D}_R(M) = \{ D \in \mathcal{D}_{\mathbb{K}}(M \otimes_R \mathbb{K}) : D \cdot M \subseteq M \}. 
\end{equation}
These constructions naturally sheafify to sheaves of vector spaces $\mathcal D^m(\mathcal M,\mathcal N)$ and $\mathcal D(\mathcal M,\mathcal N)$ on a $\bbc$-scheme $X$, 
with $\mathcal{D}(M)$ becoming a sheaf of $\bbc$-algebras. Denote
\[
\mathcal{D}^m(X) := \mathcal{D}^m({\mathcal{O}}_X) \quad \mbox{and} \quad \mathcal{D}_{\mathcal{L}}^m(X) := \mathcal{D}^m({\mathcal{L}}). 
\]
For all curves $X$ except the framed ones, the isomorphism class of the algebra of differential operators $\mathcal{D}(X)$ determines the isomorphism 
class of $X$, \cite{Makar-Limanov}. For framed curves we have:

\begin{theorem} {\em{(i)}} {\em{(}}Berest--Wilson \cite{BW04}{\em{)}} For each line bundle $\mathcal L$ on a framed curve $X$, there exists $n \in \bbn$ such that
\[
\mathcal{D}_{\mathcal{L}}(X) = \mathcal{D}(X_n),
\]
where  $X_n = \spec(\bbc\oplus z^{n+1}\bbc[z])$ is the curve from Examples \ref{ex2.1} and \ref{ex2.4}.

{\em{(ii)}} {\em{(}}Letzter--Makar-Limanov \cite{LML}{\em{)}} For $m,n \in \bbn$, $\mathcal D(X_m) \cong\mathcal D(X_n)$ if and only if $m=n$.
\end{theorem}
The special case of part (i) the theorem for $\mathcal{L} = \mathcal{O}_X$ was proved by Kouakou \cite{kouakou}. 
By a result of Smith and Stafford \cite{Smith}, all algebras $\mathcal{D}_{\mathcal{L}}(X)$ are Morita equivalent to each other.

\begin{definition}
The \vocab{differential genus} of a line bundle $\mathcal L$ on a framed curve $X$ is defined to be the unique $n \in \bbn$ 
satisfying $\mathcal D_{\mathcal L}(x) \cong\mathcal D(X_n)$.
\end{definition}
\begin{example}
\label{ex2.8} Recall from Example \ref{ex2.5} that the line bundles associated to the Bessel planes $W_{\Be(n)} \in \Grad$ are trivial. 
One easily verifies that the corresponding curves $X_{\Be(n)}$ have differential genus $n$, i.e. that $\mathcal{D}(X_{\Be(n)}) \cong \mathcal{D}(X_n)$. 
\qed
\end{example}

\section{The growth of Fourier algebras via algebras of differential operators on line bundles}\label{sec:Thm E}
In his section we develop an algebro-geometric framework to obtain sharp estimates on the growth of the left and right Fourier algebras 
$\mathcal F_x(\psi)$, $\mathcal F_z(\psi)$ of all rank $1$ bispectral
functions $\psi(x,z)$, and use them to prove Theorem E from the introduction.

\subsection{Fourier algebras of bispectral functions and bifiltrations}
\label{3.1}
Recall that to each bispectral meromorphic functions $\psi(x,z)$ on a subdomain 
$\Omega_1 \times \Omega_2 \subseteq \bbc^2$ we associated the left and right 
Fourier algebras $\mathcal{F}_x(\psi)$ and $\mathcal{F}_z(\psi)$ (defined in \eqref{Fx}--\eqref{Fz})
and an antiisomorphism $b_\psi : \mathcal{F}_x(\psi) \to \mathcal{F}_z(\psi)$ 
(defined in \eqref{bpsi}). We will call the later the {\em{generalized Fourier map}} 
associated to $\psi(x,z)$. 

The left Fourier algebra $\mathcal{F}_x(\psi)$
has an $\bbn \times \bbn$-filtration given by \eqref{Fx-filt}. The right 
Fourier algebra $\mathcal{F}_z(\psi)$ has the $\bbn \times \bbn$-filtration
\[
\mathcal F_z^{\ell,m}(\psi) := b_\psi( \mathcal{F}_x^{m, \ell}(\psi)) = 
\{R(z,\partial_z)\in\mathcal F_z(\psi):\ord \, R(z,\partial_z)  \leq \ell,\ \ord \, (b_\psi^{-1}R)(x,\partial_x)) \leq m \}, 
\]
where $\ell, m \in \bbn$. The algebras
\[
\mathcal{B}_x(\psi) := \bigcup_{\ell \in \bbn} \mathcal{F}_x^{\ell, 0}(\psi)
\quad \mbox{and} \quad
\mathcal{B}_z(\psi) := \bigcup_{\ell \in \bbn} \mathcal{F}_z^{\ell, 0}(\psi)
\]
are precisely the bispectral algebras of $\psi(x,z)$, i.e. the algebras of differential operators 
 for which $\psi(x,z)$ is an eigenfunction in $x$ and $z$, respectively. The proof of the following lemma 
is simple and is left to the reader.
\begin{lemma} For all $W \in \Grad$, the algebras
\[
\bigcup_{\ell \in \bbn} \mathcal{F}_z^{0, \ell}(\psi) \stackrel{b_{\psi_W}^{-1}}{\cong} \mathcal{B}_x(\psi) 
\quad \mbox{and} \quad
\bigcup_{\ell \in \bbn} \mathcal{F}_x^{0, \ell}(\psi)  \stackrel{b_{\psi_W}}{\cong} \mathcal{B}_z(\psi)
\]
are canonically identified with $A_W$ and $A_{b(W)}$. 
\end{lemma}

\begin{example} Consider the Bessel wave function $\psi_{\Be(n)}(x,y)$ from Examples \ref{ex2.2} and \ref{ex2.5} for $n \in \bbn$.
They are fixed by the bispectral involution $b$ of $W$, and as a consequence, 
the spaces $\mathcal F_z^{\ell,m}(\psi_{\Be(n)})$ for the right Fourier algebras 
are obtained from the spaces $\mathcal F_x^{\ell,m}(\psi_{\Be(n)})$ for the 
left Fourier algebras by replacing $x$ with $z$. 

Let $Q_n = \{1,3,\dots, 2n-1\}$.
The left Fourier algebra has the simple expression
$$\mathcal F_x(\psi_{\Be(n)}) = \vspan_\bbc\{x^rf(x\partial_x): f(k-n)=0\ \forall k\in \bbz_+\diff Q_n,\ \text{with}\ k+r\notin \bbz_+\diff Q_n\}.$$
When $n=0$, this is precisely the Weyl algebra.

The bifiltration of the left Fourier algebra is given by
$$\mathcal F_x^{\ell,m}(\psi_{\Be(n)}) = \Span \{x^rf(x\partial_x)\in \mathcal F_x(\psi_{\Be(n)}): \deg(f)+r\leq m,\ \deg(f)\leq \ell\}.$$
In particular, we see that $\mathcal F_x^{\ell,m}(\psi_{\Be(n)})$ contains an element of order $\ell$ with leading coefficient $x^m$ if and only if $\ell+m\notin Q_n$ and thus
$$\dim F^{\ell,m}(\psi_{\Be(n)})\geq (\ell+1)(m+1)-n(n+1),$$
with equality when $\ell>n$ and $m\geq n$.
In particular the differential genus of the associated line bundle is $n(n+1)/2$.
\end{example}

\begin{example}\label{first family algebra}
For a fixed $a$, consider the rank $1$ bispectral function
$$\psi(x,z) = \frac{(x-z^{-1})^3 + z^{-3} + a}{x^3 + a}e^{xz},$$
which is not fixed by the bispectral involution of $\Grad$. 
The point of Wilson's adelic Grassmannian associated to $\psi(x,z)$ is
\begin{align*}
W &= \Span\left\lbrace \partial_x^k|_{x=0}\psi(x,z) : k \in \bbn \right\rbrace = \Span\left\lbrace 1, z + 3a^{-1} z^{-2}, z^2, z^3, \dots\right\rbrace.
\end{align*}
In this case
\begin{align*}
\mathcal A_W &= \Span\{1,z^2,z^4,z^5,z^6,\dots\},\\
\mathcal F_z(\psi) &= {\mathcal{D}}(W) = \Span\{S_{\ell,m}(z,\partial_z): \ell,m\geq 0,\ \ell+m\neq 1,3\},
\end{align*}
where
$$S_{\ell,m}(z,\partial_z) = \ad_{z^2}^m(S_{\ell+m,0}(z,\partial_z))$$
with
\begin{align*}
S_{2j+5k,0}(z,\partial_z) &= \left[\partial_z^2  + 2a z^2\partial_z + 2a z - \frac{6}{z^2}\right]^j\left[\partial_z^5 - \frac{15}{z^2}\partial_z^3 - \left(\frac{5}{2}a - \frac{45}{z^3}\right)\partial_z^2 - \frac{45}{z^4}\partial_z\right]^k.
\end{align*}
The associated filtered components of the right Fourier algebra are
$$\mathcal F^{\ell,m}_z(\psi) = \Span\{S_{j,k}(z,\partial_z): 0\leq j\leq \ell,\ \ 0\leq k\leq m,\ \ j+k\neq 1,3\}.$$

\end{example}
\subsection{Fourier algebras, differential operators on line bundles and $W$-constraints}
\label{3.2}
Following \cite[Sect. 5]{BW04}, for $W \in \Grad$, define the ring of differential operators on $W$ 
\[
\mathcal{D}(W) := \{ D \in \bbc(z)[\partial_z] : D \cdot W \subseteq W\},
\]
where we use the canonical action of $\bbc(z)[\partial_z]$ on $\bbc(z)$. Let $(X, \mathcal L)$ be the pair 
of a framed affine curve and line bundle on it associated to $W$. The framing $\iota: \bba^1\rightarrow X$ 
identifies the function field of $X$ with $\bbc(z)$, and by \eqref{sub-diff}, $\mathcal{D}_{\mathcal{L}}(X)$ is canonically 
embedded in $\mathcal{D}( \bbc(z)) \cong \bbc(z)[\partial_z] $. Since $\mathcal L$ is the sheafification of the 
$A_W \cong \bbc[X]$-module $W$, 
we obtain at once from \eqref{sub-diff} that $\mathcal{D}_{\mathcal{L}}(X)$ is canonically identified with $\mathcal{D} (W)$,
(see also \cite[Prop. 2.6 and 5.5]{BW04}). The next theorem relates the right Fourier algebras of any rank $1$ bispectral function to 
these algebras.
\begin{theorem} 
\label{ident}
Let $W \in \Grad$ and $(X, \mathcal L)$ be the corresponding pair of a framed affine curve and line bundle on it 
as defined in \S \ref{2.2}. We have the canonical identifications:
\[
\mathcal{F}_z (\psi_W(x,z)) = \mathcal{D} (W) = \mathcal{D}_{\mathcal{L}}(X). 
\]
\end{theorem}
Applying the bispectral involution of $\Grad$, we obtain that for the left Fourier algebras of any rank $1$ bispectral function
there are the canonical identifications
\[
\mathcal{F}_x (\psi_W(x,z)) = \mathcal{D} (b(W)) = \mathcal{D}_{b(\mathcal{L})}(b(X)),
\]
where $(b(X), b(\mathcal{L}))$ is the pair of a framed affine curve and line bundle on it associated to $b(W) \in \Grad$. 
\begin{proof}[Proof of Theorem \ref{ident}]
Let $\psi_W(x,z) = e^{xz}\frac{h(x,z)}{p(x)q(z)}$ for polynomials $p(x),q(z),$ and $h(x,z)$ and set $\psi(x,z) = p(x)\psi_W(x,z)$.
Recall that
$$W = \vspan\left\lbrace\frac{\partial^j}{\partial x^j}(\psi(x,z))|_{x=0}: j=0,1,\dots\right\rbrace.$$
If $R(z,\partial_z)\in \mathcal F_z(\psi_W)$, then
$$R(z,\partial_z)\cdot \psi(x,z) = \sum_{k=0}^d a_k(x)\frac{\partial^k \psi(x,z)}{\partial x^k},\ \text{for}\ p(x)b_\psi^{-1}(R(z,\partial_z))p(x)^{-1} = \sum_{k=0}^d a_k(x)\partial_x^k.$$
It follows that
$$R(z,\partial_z)\cdot\frac{\partial^j}{\partial x^j}\psi(x,z) = \sum_{\ell=0}^{j+k} \sum_{k=0}^d \binom{j}{k+j-\ell} a_k^{(k+j-\ell)}(x)\frac{\partial^{\ell} \psi(x,z)}{\partial x^{\ell}}.$$
Thus $R(z,\partial_z)\cdot W \subseteq W$ and this proves $\mathcal F_z(\psi_W)\subseteq \mathcal D(W)$.

Conversely, suppose that $R(z,\partial_z)\in\mathcal D(W)$.
Choose $d>0$ such that
$$R(z,\partial_z)\frac{\partial^k}{\partial x^k}\psi(x,z)|_{x=0} \in \vspan\left\lbrace\frac{\partial^j}{\partial x^j}(\psi(x,z))|_{x=0}: j=0,1,\dots, d+k\right\rbrace.$$
Then for all $j\geq 0$ there exist constants $c_{j\ell}$ for $0\leq \ell\leq j+d$ such that
$$R(z,\partial_z)\cdot\frac{\partial^j}{\partial x^j}\psi(x,z) = \sum_{\ell=0}^{j+d} c_{j\ell}\frac{\partial^{\ell} \psi(x,z)}{\partial x^{\ell}}.$$
Define functions $a_0(x),\dots, a_d(x)$ whose Taylor series recursively satisfy
$$c_{j\ell} = \sum_{k=0}^d \binom{j}{k+j-\ell} a_k^{(k+j-\ell)}(0), 0\leq \ell\leq j+k.$$
Then 
$$R(z,\partial_z)\cdot \psi(x,z) = \sum_{k=0}^d a_k(x)\frac{\partial^k \psi(x,z)}{\partial x^k},$$
and therefore $R(z,\partial_z)\in \mathcal F_z(\psi_W)$.
This proves $\mathcal D(W)\subseteq\mathcal F_z(\psi_W)$.
\end{proof}

\begin{remark} 
\label{rmk-incarnations}
The $W_{1 + \infty}$-algebra, introduced in \cite{KP}, is the central extension of $w_\infty := \bbc[z^{\pm 1}][\partial_z]$ given by
\[
[ z^\ell e^{s z \partial_z}, z^m e^{t z \partial_z}] := (e^{s m} - e^{t \ell}) z^{ \ell +m} e^{(s+t) z \partial_z} + \delta_{\ell, -m} \frac{e^{s m} - e^{t \ell}}{1- e^{s+t}} c
\]
in terms of generating functions in $s$ and $t$. The two algebras are the algebras of additional symmetries of the KP hierarchy, 
written in terms of actions on tau functions and wave functions, respectively, see \cite[Sect. 4 and 5]{vanMoerbeke}. 

The subalgebra of $W_{1+ \infty}$ fixing a given tau function of the KP hierarchy is called the algebra of  
$W$-constraints of the tau function. For a given $W \in \Grad$, the algebra of the $W$-constraints of the corresponding 
tau function $\tau_W$ is precisely the central extension of the algebra of differential operators $\mathcal{D}(W)$ on $W$. 
In this way, for all $W \in \Grad$, Theorem \ref{ident} provides two other incarnations of the right Fourier algebra $\mathcal{F}_z(\psi_W(x,z))$ 
as the algebra of the $W$-constraints of the tau function $\tau_W$ (without the central extension) and the 
algebra of differential operators on the line bundle $\mathcal{L}$ on the framed affine curve $X$ associated to $W$. 
Theorem E describes the exact size of the $W$-constraints of the tau functions associated to 
all points of the adelic Grassmannian. A geometric interpretation of the $W$-symmetries of the adelic Grassmannian was given by 
Ben-Zvi and Nevins in \cite{BZN}.
\end{remark}
Let $\varphi: \mathcal L_\eta\rightarrow\mathcal O_{X,\eta}$ be the local trivialization of the line bundle $\mathcal L$ of a point $W\in\Grad$
introduced in \S \ref{2.2}.
The bifiltration of $\mathcal F_z(\psi)$ is naturally induced by $\varphi$ in the following way.
First note that postcomposing $\varphi$ with the natural map $\iota_\eta: \mathcal O_{X,\eta}\rightarrow \mathcal O_{\bba^1_\bbc,\eta} = \bbc(z)$ 
defines an injection $\mathcal L(X)\rightarrow \bbc(z)$, and thereby a filtration of $\mathcal L(X)$ by degree.
Let 
\[
\mathcal L_n(X) = \iota_\eta(\varphi(\mathcal L(X))) \cap \{f(z)\in\bbc(z): \deg(f(z))\leq n\}.
\]
\begin{proposition}
$$\mathcal F_z^{\ell,m}(\psi) = \{R(z,\partial_z)\in\mathcal D^\ell(\psi): R(z,\partial_z)\cdot \mathcal L_n(X)\subseteq \mathcal L_{n+m}(X)\ \forall n \in \bbn\}.$$
\end{proposition}
\begin{proof}
From the proof of Theorem \ref{ident}, we see that $R(z,\partial_z)\in \mathcal F_z^{\ell,m}(\psi)$ if and only if $R(z,\partial_z)\in \mathcal D^\ell(\psi)$ and
$$R(z,\partial_z)\frac{\partial^k}{\partial x^k}\psi(x,z)|_{x=0} \in \vspan\left\lbrace\frac{\partial^j}{\partial x^j}(p(x)\psi(x,z))|_{x=0}: j=0,1,\dots, m+k\right\rbrace.$$
for some polynomial $p(x)$.
The statement of the Proposition follows immediately.
\end{proof}
\subsection{Geometric estimates of the sizes of Fourier algebras}
Recall Wilson's decomposition of the adelic Grassmannian \eqref{CM}. The $k$-th stratum is isomorphic to
\[
{\mathcal{CM}}_k \cong \{ (X, Z) \in {\mathrm{Mat}}^2_{k \times k} (\bbc) : \mathrm{rank} ([X,Z] + I_k) = 1 \} / GL_k(\bbc),
\]
where the action of $GL_k(\bbc)$ is faithful and the quotient is geometric. The following results were proved in 
\cite[Thm. 5.6 and Prop. 5.5]{BW04}.

\begin{theorem}[Berest--Wilson] 
\label{th3.7}
For each $W \in \Grad$ the differential genus of the associated pair $(X, \mathcal{L})$ of a framed curve and line bundle on it, 
equals the integer $k \in \bbn$ such that $W \in {\mathcal{CM}}_k$. 
\end{theorem} 
In this setting, let
\[
g(W):=k.
\]
This number can be calculated directly from $W$ as follows (see Theorem 8.1 in \cite{BW04} and the discussion after it).
\begin{proposition}
Let $W\in\Grad$ and let $p(z)\in \bbc[z]$ with $p(z)W\subseteq\bbc[z]$.
For each $a\in\bbc$, define $r_{0,a}<r_{1,a}<\dots$ such that 
$$\{r_{i,a}\}_{i=0}^\infty = \{n\in\bbz_+:\ \text{$\exists\ g(z)\in p(z)W$ vanishing to order exactly $n$ at $a$}\}.$$
There exists $m_a\in\bbz_+$ such that $r_{i,a} = m_a+i$ for all $i\gg0$ and
$$g(W) = \sum_{a\in\bbc} g_a(W),\ \ \text{where}\ \ g_a(W) = \sum_{i\geq 0} m_a+i-r_{i,a}.$$
\end{proposition}

\begin{example}
\label{ex3.9}
Consider the point $W\in\Grad$ as in Example \ref{first family algebra} so that
\begin{align*}
z^2W = \Span\left\lbrace z^2, z^3 - \frac{3}{a}, z^4, z^5, \dots\right\rbrace.
\end{align*}
For all $a\in\bbc$ with $a\neq 0$ we have $r_{i,a} = i$ so that $g_a(W) = 0$.
Moreover $r_{0,0} = 0$, $r_{1,0} = 2$, and $r_{i,0} = i+2$ for $i>1$, so that we have
$$g(W) = g_0(W) = \sum_{i\geq 0} m_0+i-r_{i,0} = 3.$$
\end{example}

Our estimates for the size of the Fourier algebra in this section are based the relationship between the differential genus of $(X, \mathcal{L})$ and the Letzter--Makar-Limanov invariant \cite{LML} 
of $(X, \mathcal{L})$.  To define the latter, we first note the following result. Recall from \S \ref{3.2} that $\mathcal{D}_{\mathcal{L}}(X)$ is canonically 
embedded in $\mathcal{D}( \bbc(z)) \cong \bbc(z)[\partial_z]$ via the framing $\iota: \bba^1\rightarrow X$.
\begin{theorem}[Letzter--Makar-Limanov \cite{LML}] For each affine framed curve $X$ and a line bundle $\mathcal{L}$ on it 
the vector space
$$\{p(z)x^j: \text{$\exists R(z,\partial_z)\in {\mathcal{D}}_{\mathcal{L}}(X)$ with $\ord(R) = j$ and with leading coefficient $p(z)$}\},$$
has finite codimension in $\bbc[z,x]$.
\end{theorem}
\begin{definition}
The codimension from the previous theorem is called the \vocab{Letzter--Makar-Limanov invariant} $\LM(X, \mathcal{L})$ of $(X, \mathcal{L})$.
\end{definition}

The Letzter-Makar-Limanov invariant provides another way of measuring the differential genus of $W$.
\begin{theorem}[\cite{BW04} Theorem 8.6]
Let $W\in \Grad$ and $(X, \mathcal{L})$ be the associated pair of a framed curve and line bundle on it.
Then $g(W)= \LM(X, \mathcal{L}) /2$.
\end{theorem}

\begin{example}
Consider the point $W\in\Grad$ as in Example \ref{ex3.9}.
From our explicit expression for $\mathcal{D}(W)$ (which by Theorem \ref{ident} is identified with $\mathcal{D}_{\mathcal{L}}(X)$), 
we see that $\mathcal{D}(W)$ contains an operator of order $\ell$ with leading coefficient of degree $m$ for all $\ell\geq 0$ and 
$m\geq 0$ with $\ell+m\neq 1,3$ and therefore $\LM(X,\mathcal{L}) = 6$. This gives a second proof of the fact that 
$g(W) = 3$, verified by a different method in Example \ref{ex3.9}.
\end{example}
Recall the adjoint and bispectral involutions of $\Grad$ from \S \ref{1.3} and \S \ref{2.1}. 
\begin{proposition} 
\label{abW}
For all $W\in\Grad$,
\[
g(b W ) = g(a W ) = g(W). 
\]
\end{proposition}
\begin{proof} Denote by $(X, {\mathcal{L}})$ and $(a(X) , a({\mathcal{L}}))$ the pairs of a line bundle and framed curve
associated to $W$ and $a W $. By Theorem \ref{ident}, 
$\mathcal{D}_{\mathcal{L}}(X) = \mathcal{D} (W)$ and $\mathcal{D}_{a(\mathcal{L})}( a(X) ) = \mathcal{D}(a W )$.
It follows from \eqref{aW} that 
\[
\mathcal{D}(a W) = ( \mathcal{D}(W))^\dag
\]
where $(\cdot)^\dag$ denotes the formal adjoint of a differential operator with rational coefficients, as in \eqref{formal adjoint formula}. 
Therefore, $\LM(a(X) , a({\mathcal{L}})) = \LM (X, {\mathcal{L}})$, and hence, $g( a W) = g(W)$. 
The composition 
\[
\mathcal{D}(a W)\stackrel{(\cdot)^\dag}{\rightarrow} \mathcal{D}(W) \stackrel{b_\psi}{\rightarrow}\mathcal{D}(b W )
\]
is an algebra isomorphism, and using one more time Theorem \ref{ident} proves that $g(b W) = g(W)$.

Alternatively, the proposition also follows by using the facts that the adjoint and bispectral involutions 
preserve the strata ${\mathcal{CM}}_k$ of $\Grad$ (see \cite{Wilson-98}) and then applying Theorem \ref{th3.7}.
\end{proof}
We finally come to the proof of Theorem E part (i).
\medskip

\noindent
{\em{Proof of Theorem E part (i).}} Recall from \S \ref{2.1} that there exist differential operators $P(x,\partial_x)$, $\wt P(x,\partial_x)$,
$Q(z,\partial_z)$, $\wt Q(z,\partial_z)$ 
with polynomial coefficients and polynomials $p(x),\wt p(x),q(z),\wt q(z)$ such that
$$\psi_W(x,z) = \frac{1}{q(z)p(x)}P(x,\partial_x)\cdot \psi_{\exp}(x,z) = \frac{1}{q(z)p(x)}Q(z,\partial_z)\cdot \psi_{\exp}(x,z)$$
and
$$\wt P(x,\partial_x)\frac{1}{p(x)\wt p(x)}P(x,\partial_x) = q(\partial_z)\wt q(\partial_z)\ \ \text{and}\ \ 
\wt Q(z,\partial_z)\frac{1}{\wt q(z)q(z)}Q(z,\partial_z) = p(\partial_x)\wt p(\partial_x).$$
Without loss of generality, we choose the above presentation of $\psi_W$ so that the inequalities of \eqref{equalities} become equalities.

For any differential operator $L(x,\partial_x)$ with polynomial coefficients, we have
$$\frac{1}{p(x)}P(x,\partial_x)L(x,\partial_x)\wt P(x,\partial_x)\frac{1}{\wt p(x)}\cdot \psi_W(x,y) = \wt q(z)b_{\psi_{\exp}}(L(x,\partial_x))q(z)\cdot\psi_W(x,y).$$
Thus for all $\ell\geq 0$ and $m\geq \deg(W)+\deg(aW)$ the Fourier algebra $\mathcal F_z^{\ell,m}(\psi_W)$ contains a differential operator of order $\ell$ with leading coefficient a polynomial of degree $m$.

For each $\ell\geq 0$ and $m\geq 0$, let $R_{\ell,m}(z,\partial_z)\in\mathcal F_z(\psi_W)$ be a differential operator of order $\ell$ with leading coefficient a polynomial of degree $m$ (if it exists), chosen such that the order of $b_{\psi_W}^{-1}(R_{\ell,m})$ is as small as possible.
Note that the order of $b_{\psi_W}^{-1}(R_{\ell,m})$ is equal to the maximum of the degrees of the coefficients of $R_{\ell,m}(z,\partial_z)$ (as can be seen by the Laurent series expansion at $\infty$), and therefore the order of $b_{\psi_W}^{-1}(R_{\ell,m})$ is at least $m$.

We claim that for all $\ell,m$ if $\ord b_{\psi_W}^{-1}(R_{\ell,m})\neq m$, then $m\leq \deg(W)+\deg(aW)-1$.
To see this, suppose otherwise.
That is, let $R_{\ell,m}(z,\partial_z)$ be a differential operator with $m\geq \deg(W) + \deg(aW)$ and with $\wt m := \ord b_{\psi_W}^{-1}(R_{\ell,m})\neq m$.
Then we may write
$$R_{\ell,m}(z,\partial_z) = a_\ell(z)\partial_z^\ell + \sum_{j=1}^r c_jz^{\wt m}\partial_z^{k_j} + \sum_{j=0}^{\ell-1}a_j(z)\partial_z^j$$
for $a_\ell(z)$ a polynomial of degree $m$, $c_1,\dots, c_r\in\bbc$ and $a_0(z),\dots,a_{\ell-1}(z)\in \bbc(z)$ with $\deg(a_j(z)) < \wt m$.
For each $j$ let $S_j(z,\partial_z)\in \mathcal F_z(\psi_W)$ be an operator of order $k_j$ with leading coefficient a monic polynomial of degree $\wt m$, which exists by the fact that $\wt m\geq m+1 = \deg(W)+\deg(aW)$.
Then the operator
$$\wt R(z,\partial_z) := R_{\ell,m}(z,\partial_z) - \sum_{j=1}^r c_jS_j(z,\partial_z)$$
has order $\ell$ with leading coefficient of degree $m$, but the order of $b_{\psi_W}^{-1}(\wt R)$ is strictly less that $\wt m$.
This contradicts the definition of $R_{\ell,m}(z,\partial_z)$ and proves our claim.

Thus for all $\ell\geq 0$ and $m\geq \deg(W)+\deg(aW)$ we have 
\begin{align*}
\dim\mathcal F_z^{\ell,m}(\psi_W)
  & \geq \dim \Span\{R_{j,k}(z,\partial_z): 0\leq j\leq \ell,\ 0\leq k\leq m\}\\
  & \geq (\ell + 1)(m+1) - \LM(X, \mathcal{L})\\
  & = (\ell + 1)(m+1) - 2g(W).
\end{align*}
This argument with $W$ replaced by $bW$ gives us the same estimation for $m\geq 0$ and $\ell\geq \deg(bW) + \deg(abW)$.
Finally, since the order of $b_{\psi_W}^{-1}(R_{\ell,m})$ is at least $m$ we obtain that
$$\mathcal F_z^{\ell,m}(\psi_W)\subseteq \Span\{R_{j,k}(z,\partial_z): 0\leq j\leq \ell,\ 0\leq k\leq m\},$$
and therefore, for $\ell\geq \deg(W) + \deg(aW)$ and $m\geq \deg(bW) + \deg(abW)$, we have
$$\dim\mathcal F_z^{\ell,m}(\psi_W) \leq \dim\Span\{R_{j,k}(z,\partial_z): 0\leq j\leq \ell,\ 0\leq k\leq m\} = (\ell+1)(m+1)-2g(W).$$
This completes the proof of Theorem E part (i).
\qed
\medskip

\noindent
{\em{Proof of Theorem E part (ii).}}
We will prove that 
\begin{equation}
\label{ineq1}
g(W)\leq \deg(W)\deg(aW).
\end{equation}
The inequality $g(W) \leq \deg(bW)\deg(abW)$ follows from it and an application of the bispectral map. 

There exists $P(x,\partial_x),\wt P(x,\partial_x)\in\bbc[x,\partial_x]$,  $p(x),\wt p(x)\in\bbc[x]$, and $\wt q(z),q(z)\in\bbc[z]$ such that
$$\psi_W(x,z) = \frac{1}{q(z)p(x)}P(x,\partial_x)\cdot \psi_{\exp},\ \ \wt P(x,\partial_x)\frac{1}{\wt p(x)p(x)}P(x,\partial_x) = g(\partial_x) := q(\partial_x)\wt q(\partial_x).$$
The orders $m$ and $\wt m$ of $P(x,\partial_x)$ and $\wt P(x,\partial_x)$ are the degree of $W$ and $aW$, respectively for $W\in \Grad$ associated to $\psi_W(x,z)$.
Explicitly we can write
$$q(z)W = V_C = \{f(z)\in\bbc[z]: \langle f(z),\chi(z)\rangle = 0\ \forall \chi\in C\}$$
where we have
$$C= \{\chi: \langle e^{xz},\chi(z)\rangle\in\ker P(x,\partial_x)\}.$$
Since $P(x,\partial_x)$ right divides $g(\partial_x)$, we know that
$$C\subseteq \{\chi: \langle e^{xz},\chi(z)\rangle\in\ker g(\partial_x)\} = \{\delta^{(k)}(z-a):a\in\bbc,\ (z-a)^k\ \text{divides}\ g(z) \}.$$
This decomposes as $C = \bigoplus_{g(a)=0} C_a$ with
$$C_a = C\cap \mathscr C_a \subseteq \{\delta^{(k)}(z-a): (z-a)^k\ \text{divides}\ g(z) \}.$$
For each root $a$ of $g(z)$ let $n_a$ be the multiplicity and $m_a = \dim C_a$.
Then
$$g_a(W) = \sum_{i\geq 0} m_a + i-r_{i,a} \leq \sum_{i=0}^{n_a-m_a-1} m_a+i-i = (n_a-m_a)m_a.$$
Thus
$$g(W) = \sum_{g(a)=0} g_a(\psi)\leq \sum_{g(a)=0} (n_a-m_a)m_a \leq \left(\sum_{g(a)=0}(n_a-m_a)\right)\left(\sum_{g(a)=0}m_a\right) = \wt mm$$
This completes the proof.
\qed
\section{Reflective integral operators and commuting pairs}\label{sec:Thm A-D}
\subsection{Time-band limiting operators for generalized Laplace and Fourier transforms}
Throughout this section, let $\tau>0$ and $\psi_W(x,y)$ be a rank $1$ bispectral meromorphic function corresponding to a point $W\in\Grad$, whose poles lie in the complement of $[\tau,\infty)\times[\tau,\infty)$.
For any integer $\ell>0$, we define integral operators $L_W$ and $F_W$ on the Sobolev space $W^{\ell,2}([\tau,\infty))$ generalizing the Laplace and Fourier transforms, along with their adjoints $L_W^*,F_W^*$ by
\begin{align*}
L_W: \ \ f(x)\mapsto \int_\tau^\infty \psi_W(y,-z)f(y)dy,\ \ &L_W^*: f(z)\mapsto \int_\tau^\infty \ol{\psi(x,-w)}f(w)dw\\
F_W: \ \ f(x)\mapsto \int_\tau^\infty \psi_W(y,iz)f(y)dx,\ \ &F_W^*: f(z)\mapsto \int_\tau^\infty \ol{\psi_W(x,iw)}f(w)dw
\end{align*}
As we remarked in the introduction, when $\tau = 0$ and $W = \bbc[z]$ the operator $L_W$ is precisely the Laplace transform.
The proof of the following proposition is simple and is left to the reader.
\begin{proposition}
The linear maps $F_W, F_W^*, L_W$ and $ L_W^*$ are bounded operators on $W^{\ell,2}([\tau,\infty))$.
\end{proposition}

Fix $r,t\in[\tau,\infty)$.
The time and band-limiting operators of $L_W$ and $F_W$ above are the operators
$E_W,\mathcal E_W :L^2([\tau,\infty ))\to L^2([\tau,\infty ))$ given by 
$$
\mathcal E_W(f(x))(z) = \chi_{[t,\infty)}(z)L_W(\chi_{[r,\infty)}(x)f(x)),\ \ E_W(f(x))(z) = \chi_{[t,\infty)}(z)F_W(\chi_{[r,\infty)}(x)f(x)).
$$
The corresponding adjoints are operators $L^2([t,\infty))\rightarrow L^2([r,\infty))$ defined similarly, using the adjoints $L_W^*$ and $F_W^*$.
Here the multiplication by the characteristic functions $\chi_{[s,\infty)}$ and $\chi_{[t,\infty)}$ corresponds to the time-limiting operator and the band-limiting operator, respectively.

In time and band-limiting, we consider the self-adjoint operators $\mathcal E_W\mathcal E_W^*$ and $E_WE_W^*$ as in \eqref{tblimit laplace equation} and \eqref{tblimit fourier equation}, with the specific goal of computing their eigenfunctions.
As mentioned in the introduction, the operator $\mathcal E_W\mathcal E_W^*$ has a kernel-type operator expression
\begin{equation}
(\mathcal E_W\mathcal E_W^*f)(z) = \chi_{[t,\infty)}(z)\int_t^\infty K_W(z,w)f(w)dy,\ \ K_W(z,w) = \int_r^\infty \psi_W(x,-z)\ol{\psi_W(x,-w)}dx.
\end{equation}
However, a similar expression for $E_WE_W^*$ has to be understood as a distribution in the sense of Laurent Schwartz, see \cite{DK10,GS64,H03}. 
The reader may benefit from reading the introduction to \cite{S66}. 
It is remarkable that M. Sato's theory of Hyperfunctions deals with many of the same problems addressed by distribution theory, see \cite{KKK86}.

In time and band-limiting, the strategy for finding eigenfunctions of the above operator is to determine a differential operator commuting with the integral operator.
In such a case, the corresponding eigenfunctions are obtained by solving numerically the associated differential equation.
The behavior of a differential operator when composed with an integral operator is described by means of integration by parts, giving rise naturally to the bilinear concomitant.
For a differential operator $P=P(x,\partial_x) := \sum_{j=0}^m a_j(x)\partial_x^j$ the bilinear concomitant is defined by
\begin{align}
\mathcal C_{P}(f,g;p)
  & = \sum_{j=1}^m \sum_{k=0}^{j-1} (-1)^k f^{(j-1-k)}(x)(a_j(x)g(x))^{(k)}|_{x=p}\nonumber\\
  & = \sum_{j=1}^m \sum_{k=0}^{j-1}\sum_{\ell=0}^k\binom{k}{\ell} (-1)^k f^{(j-1-k)}(x)a_j(x)^{(k-\ell)}g(x)^{(\ell)}|_{x=p},
\end{align}
where here $p\in\bbc$ and $f,g$ are two functions which are sufficiently smooth in a neighborhood of $p$.
The bilinear concomitant generalizes the formula for integration by parts in that
\begin{equation}\label{integral biconcomitant}
\int_{x_0}^{x_1}[P(x,\partial_x)\cdot f(x))g(x)-f(x)(P(x,\partial_x)^\dag\cdot g(x))] dx = \mathcal C_{P}(f,g;x_1)-\mathcal C_{P}(f,g;x_0),
\end{equation}
for any functions $f(x)$ and $g(x)$ which are sufficiently smooth and integrable.

\subsection{Reflective integral operators}
We define an operator $T_W$, related to the integral operator $\mathcal E_W\mathcal E_W^*$ in the previous section by
$$T_W: f(z) \mapsto \int_{\Gamma_1}\int_{\Gamma_2} \psi_W(x,-z)\psi_W^\dag(x,-{\xi})f({\xi})d\xi dx,$$
where here $\psi_W^\dag := a\psi_W$.
This may be rewritten as
$$T_W: f(z) \mapsto\int_{\Gamma_2} K_W(z,{\xi})f({\xi})d\xi ,\ \ K_W(x,{\xi}) = \int_{\Gamma_1} \psi_W(x,-z)\psi_W^\dag(x,-{\xi}) dx.$$
As we demonstrate below, the above integral operator reflectively commutes with a nonconstant differential operator, i.e. we have
$$R(z,\partial_z)\cdot T_W(f({\xi}))(z) = T_W(R(-\xi,-\partial_{\xi})\cdot f({\xi}))(z).$$
Moreover, in the specific case that $a\psi_W = c\psi_W$ we have $T_W=\mathcal E_W\mathcal E_W^*$ and $R(-z,-\partial_z) = R(z,\partial_z)^\dag$, so that $T$ maps the eigenfunctions of $R(z,\partial_z)^\dag$ to those of $R(z,\partial_z)$. 

\begin{proposition}\label{reflect with the kernel}
Let $\Gamma_1\in\bbc$ be a path from $r$ to $\infty$ whose real values are bounded below and imaginary values are bounded, and such that the poles in $x$ of $\psi_W(x,-z)$ 
avoid $\Gamma_1$.
Consider the function
$$K_W(z,{\xi}) := \int_{\Gamma_1} \psi_W(x,-z)\psi_W^\dag(x,-{\xi}) dx.$$
If $R = R(z,\partial_z)\in\mathcal F_z(\psi_W)$  satisfies $\mathcal C_{b_{\psi_W^{-1}}R}(f,g;r) = 0$ for all smooth functions $f(x),g(x)$ in a neighborhood of $r$, then
$$R(-z,-\partial_z)\cdot K_W(z,{\xi}) = R(\xi,\partial_{\xi})^\dag\cdot K_W(z,{\xi}).$$
\end{proposition}
\begin{proof}
First note that
$$b_{a\psi_W}((b_{\psi_W}^{-1}(R))^\dag)(\xi,\partial_{\xi}) = R(-\xi,-\partial_{\xi})^\dag.$$
By integration by parts
\begin{align*}
R(-z,-\partial_z)\cdot K_W(z,{\xi})
  & = \int_{\Gamma_1} R(-z,-\partial_z)\cdot\psi_W(x,-z)\psi_W^\dag(x,-{\xi}) dx\\
  & = \int_{\Gamma_1} b_{\psi_W}^{-1}(R)(x,\partial_x)\cdot\psi_W(x,-z)\psi_W^\dag(x,-{\xi}) dx\\
  & = \int_{\Gamma_1} \psi_W(x,-z)b_{\psi_W}^{-1}(R)(x,\partial_x)^\dag\cdot\psi_W^\dag(x,-{\xi}) dx\\
  & = \int_{\Gamma_1} \psi_W(x,-z)b_{a\psi_W}((b_\psi^{-1}(R))^\dag)(-\xi,-\partial_{\xi})\cdot\psi_W^\dag(x,-\xi) dx\\
  & = b_{a\psi_W}((b_{\psi_W}^{-1}(R))^\dag)(-\xi,-\partial_{\xi})\cdot K_W(z,{\xi})\\
  & = R^\dag(\xi,\partial_{\xi})\cdot K_W(z,{\xi})
\end{align*}
\end{proof}

\begin{proposition}\label{reflecting prop}
Let $\Gamma_1$ and $\Gamma_2$ be two paths in $\bbc$ going from $r$ and $t$, respectively to $\infty$.
Assume moreover that the real part of $\Gamma_i$ is bounded below and the imaginary part is bounded, and that the poles of $\psi_W(x,-z)$ and $\psi_W^\dag(x,-z)$ avoid $\Gamma_1\times\Gamma_2$.
If $R = R(z,\partial_z)\in\mathcal F_z(\psi)$ is an operator of order $\ell$ satisfying $\mathcal C_{b_\psi^{-1}R}(f,g;r) = 0$ and also $\mathcal C_{R(z,\partial_z)}(f,g;t) = 0$ 
for all smooth functions $f(x),g(x)$ in a neighborhood of $r$ or $t$, then for all $f\in W^{\ell+1,2}(\Gamma_2)$ with compact support we have
$$R(-z,-\partial_z)\cdot T_W(f({\xi}))(z) = T_W(R(\xi,\partial_{\xi})\cdot f({\xi}))(z).$$
\end{proposition}
\begin{proof}
By integration by parts
\begin{align*}
R(-z,-\partial_z)\cdot T_W(f({\xi}))(z)
  & = \int_{\Gamma_2} R(-z,-\partial_z)\cdot K_W(z,{\xi})f({\xi})d\xi \\
  & = \int_{\Gamma_2} R(\xi,\partial_{\xi})^\dag\cdot K_W(z,{\xi})f({\xi})d\xi \\
  & = \int_{\Gamma_2} K_W(z,{\xi})R(\xi,\partial_{\xi})\cdot f({\xi})d\xi \\
  & = T_W(R(-\xi,-\partial_{\xi})\cdot f({\xi}))(z)
\end{align*}
\end{proof}

\subsection{Commuting integral operators}\label{sec: moving to commutative}
We define an operator $U_W$ related to the operator $E_WE_W^*$, analogous to the operator $T_W$ in the previous section by
$$U_W: f(z) \mapsto \int_r^\infty\int_t^\infty  \psi_W(x,iz)\psi_W^\dag(x,-i\xi)f(\xi)d\xi dx,$$
where here $\psi_W^\dag := a\psi_W$.
The kernel of this integral operator requires a distributional interpretation as mentioned earlier. This issue is addressed in more detail in example 6.1.

The integral operator above commutes with a nonconstant differential operator
$$R(z,\partial_z)\cdot U_W(f(\xi))(z) = U_W(R(\xi,\partial_\xi)\cdot f(\xi))(z).$$
Moreover, in the specific case that $a\psi_W = c\psi_W$ we have $U_W=E_WE_W^*$.

\begin{proposition}\label{commuting prop}
If $R = R(z,\partial_z)\in \mathcal F_z(\psi_W)$ is a differential operator of order at most $\ell$ which satisfies $\mathcal C_{b_{\psi_W}^{-1}(R)}(f,g,r) = 0$ and $\mathcal C_{R(\xi,\partial_\xi)}(f,g,it) = 0$ for all functions $f,g$ smooth in a neighborood of $r$ and $it$, then
for all $f(z)\in W^{\ell+1,2}$ we have
$$R(iz,-i\partial_z)\cdot U_W(f(\xi))(z) = U_W(R(i\xi,-i\partial_\xi)\cdot f(\xi))(z).$$
\end{proposition}
\begin{proof}
As in the proof of Proposition \ref{reflect with the kernel}, we have
$$b_{a\psi_W}((b_{\psi_W}^{-1}(R))^\dag)(\xi,\partial_\xi) = R(-\xi,-\partial_\xi)^\dag.$$
By integration by parts
\begin{align*}
R(iz,-i\partial_z)\cdot U_W(f({\xi}))(z)
  & = \int_r^\infty\int_t^\infty  R(iz,-i\partial_z)\cdot\psi_W(x,iz)\psi_W^\dag(x,-i{\xi})f({\xi})d{\xi}dx\\
  & = \int_r^\infty\int_t^\infty  b_{\psi_W}^{-1}(R)(x,\partial_x)\cdot\psi_W(x,iz)\psi_W^\dag(x,-i{\xi})f({\xi})d{\xi}dx\\
  & = \int_r^\infty\int_t^\infty  \psi_W(x,iz)b_{\psi_W}^{-1}(R)(x,\partial_x)^\dag\cdot\psi_W^\dag(x,-i{\xi})f({\xi})d{\xi}dx\\
  & = \int_r^\infty\int_t^\infty  \psi_W(x,iz)b_{a\psi_W}((b_{\psi_W}^{-1}(R))^\dag)(-i{\xi},i\partial_{\xi})\cdot\psi_W^\dag(x,-i{\xi})f({\xi})d{\xi}dx\\
  & = \int_r^\infty\int_t^\infty  \psi_W(x,iz)\psi_W^\dag(x,-i{\xi})b_{a\psi_W}((b_{\psi_W}^{-1}(R))^\dag)(-i{\xi},i\partial_{\xi})^\dag\cdot f({\xi})d{\xi}dx\\
  & = \int_r^\infty\int_t^\infty  \psi_W(x,iz)\psi_W^\dag(x,-i{\xi})R(i{\xi},-i\partial_{\xi})\cdot f({\xi})d{\xi}dx\\
  & = U_W(R(i{\xi},-i\partial_{\xi})\cdot f({\xi}))(z).
\end{align*}
\end{proof}

\subsection{Proofs of Theorems A--C}
In this section, we prove Theorems A--C by using the dimension bounds of Theorem E and Propositions \ref{reflecting prop} and \ref{commuting prop}.

\begin{lemma}
Let $R(z,\partial_z)$ be a differential operator with coefficients which are holomorphic in a neighborhood $U$ of $t\in\bbc$.
If $R(z,\partial_z)$ is formally $\dagger$-symmetric (with respect to \eqref{formal adjoint formula}), then there exist functions $a_0(z),a_1(z),\dots, a_n(z)$ holomorphic on $U$ satisfying
$$R(z,\partial_z) = \sum_{j=0}^n \partial_z^ja_j(z)\partial_z^j.$$
In this case, if $a_j^{(k)}(t) = 0$ for all $0\leq k < j$ then $\mathcal C_{R}(f,g;t) = 0$ for all functions $f(z),g(z)$ analytic in a neighborhood of $t$.
If instead $R(z,\partial_z)$ is formally $\dagger$-skew symmetric, then there exist functions $b_0(z),b_1(z),\dots,b_m(z)$ holomorphic on $U$ satisfying
$$R(z,\partial_z) = \sum_{j=0}^m \partial_z^j\{b_j(z),\partial_z\}\partial_z^j.$$
In this case, if $b_j^{(k)}(t) = 0$ for all $0\leq k \leq j$ then $\mathcal C_{R}(f,g;t) = 0$ for all functions $f(z),g(z)$ analytic in a neighborhood of $t$.
\end{lemma}
\begin{proof}
The structure formulas for symmetric and skew-symmetric elements follow immediately from obvious induction arguments.
To prove the vanishing of the concomitants, recall that for any differential operators $B(z,\partial_z)$ and $C(z,\partial_z)$ that
$$\mathcal C_{BC}(f,g;z) = \mathcal C_B(C\cdot f,g;z) + \mathcal C_C(f,B^\dag\cdot g;z).$$
Consequently we have
$$\mathcal C_{\partial_z^ja_j(z)\partial_z^j}(f,g;t) = \mathcal C_{\partial_z^j}(a_jf^{(j)},g;t) + (-1)^j\mathcal C_{\partial_z^j}(f, a_jg^{(j)};t)$$
which is zero if $a_j^{(k)}(t) = 0$ for all $0\leq k < j$.
Similarly for $a_j^{(k)}(t) = 0$ for all $0\leq k \leq j$ we obtain the vanishing of the other concomitant.
\end{proof}

\begin{lemma}
Let $s\in\bbc$ and let $\mathcal F$ be a vector space of differential operators of order at most $\ell$, with coefficient which are smooth in a neighborhood of $t$.
Then for any open neighborhood $U$ of $t$ the $\bbc$-linear map
$$\mathcal C:\mathcal F\rightarrow \text{Bilinear}(C^\infty(U)\times C^\infty(U)),\ \ R(z,\partial_z)\mapsto [(f,g)\mapsto \mathcal C_{R}(f,g;t)]$$
has rank at most $d(d+1)$ if $\ell = 2d$ and $(d+1)^2$ if $\ell = 2d+1$.
\end{lemma}
\begin{proof}
First assume that $\ell = 2d$.
Then any $R = R(z,\partial_z)\in\mathcal F$ is of the form $R = R_{symm} + R_{skew}$ for some formally $\dagger$-symmetric and $\dagger$-skew symmetric operators of the form
$$R_{symm} = \sum_{j=0}^d \partial_z^ja_{2j}(z)\partial_z^j \ \ \text{and}\ \ R_{skew} = \sum_{j=0}^{d-1}\partial_z^j\{a_{2j+1},\partial_z\}\partial_z^j.$$
In particular if the $d(d+1)/2$ conditions $a_{2j}^{(k)}(t) = 0$ for all $1\leq j \leq d$ and $0\leq k< j$ and  the $d(d+1)/2$ conditions $a_{2j+1}^{(k)}(t) = 0$ for all $0\leq j < d$ and $0\leq k\leq j$ are satisfied, then the concomitant $\mathcal C_R(f,g;t)$ is zero for all $f,g$.
Thus the rank is at most $d(d+1)$.

Next assume that $\ell=2d+1$.
Then any $R = R(z,\partial_z)\in\mathcal F$ is of the form $R = R_{symm} + R_{skew}$ for some formally $\dagger$-symmetric and $\dagger$-skew symmetric operators of the form
$$R_{symm} = \sum_{j=0}^d \partial_z^ja_{2j}(z)\partial_z^j \ \ \text{and}\ \ R_{skew} = \sum_{j=0}^{d}\partial_z^j\{a_{2j+1},\partial_z\}\partial_z^j.$$
In particular if the $d(d+1)/2$ conditions $a_{2j}^{(k)}(t) = 0$ for all $1\leq j \leq d$ and $0\leq k< j$ and the $(d+1)(d+2)/2$ conditions $a_{2j+1}^{(k)}(t) = 0$ for all $0\leq j \leq d$ and $0\leq k\leq j$ are satisfied, then the concomitant $\mathcal C_R(f,g;t)$ is zero for all $f,g$.
Thus the rank is at most $(d+1)^2$.
\end{proof}

Using this lemma, we prove Theorem A and Theorem B.
\begin{proof}[Proof of Theorem A]
Let $d_1=\deg(W)$, $d_2=\deg(aW)$, $b_1=\deg(bW)$, and $b_2 = \deg(abW)$ and $d = \lfloor\min((d_1+d_2)/2,(b_1+b_2)/2)\rfloor$.
Note that
$$\min((d_1+d_2)/2,(b_1+b_2)/2)^2 \geq \min(d_1d_2,b_1b_2) + \min((d_1-d_2)/2,(b_1-b_2)/2)^2,$$
which is strict when $d_1=d_2=b_1=b_2$.
Therefore $d \geq \min(d_1d_2,b_1b_2)$.

By Proposition \ref{reflecting prop} and Proposition \ref{commuting prop}, to prove the existence of a nonconstant differential operator $R(x,\partial_x)$ of order at most $\ell$ reflectively commuting with $T_W$, it suffices to prove the space $\mathcal F^{\ell,\ell}_z(\psi_W)$ contains a nonconstant differential operator satisfying the property that the bilinear concomitants $\mathcal C_{b_{\psi_W}^{-1}(R)}(f,g;r)=0$ and $\mathcal C_{R}(f,g;t) = 0$.
Moreover, by the dimension estimates of Theorem E, we know that
$$\dim \mathcal F_z^{\ell,\ell}(\psi_W)\geq (\ell+1)^2-2d.$$
Therefore if we take $\ell = 2d$, the previous Lemma implies that $\mathcal F_z^{\ell,\ell}(\psi_W)$ contains a subspace of dimension at most $2d^2+2d+1$ satisfying the prescribed concomitant condition.  In particular it contains a nonconstant differential operator $R_{r,t}(z,\partial_z)$ depending on the finite endpoints $r,t$ of the paths $\Gamma_1,\Gamma_2$ which reflectively commutes with our integral operator.
Since $R_{r,t}(z,\partial_z)\in\mathcal F_z(\psi_W)$ it has rational coefficients in $z$.
The corresponding point in $\mathcal F_z(\psi_W)$ is found by solving a linear system of equations whose entries are polynomials in $r,t$, so $R_{r,t}(z,\partial_z)$ also depends rationally on $r$ and $t$.
\end{proof}

Applying Theorem A to the situation when $acW = W$, we obtain Theorem B.
\begin{proof}[Proof of Theorem B]
Let $\Gamma_1= [r,\infty)$ and $\Gamma_2 = [t,\infty)$.
 Then since $\psi_W^\dag(x,z) = \ol{\psi_W(\ol x, \ol z)}$, we have for real $\xi,z$ that
$$K_W(z,\xi) = \int_r^\infty \psi_W(x,-\xi)\psi_W^\dag(x,-z)dx = \int_r^\infty \psi_W(x,-\xi)\ol{\psi_W(x,-z)} dx.$$
It follows that the operator $T_W$ from Theorem A is equal to $\mathcal E_W\mathcal E_W^*$.
\end{proof}

To prove Theorem C, we cannot simply apply Theorem A due to integrability issues with the relevant kernels.  Instead, we use the results of Proposition \ref{commuting prop}.
\begin{proof}[Proof of Theorem C]
\begin{enumerate}[(i)]
\item
Let $d_1=\deg(W)$, $d_2=\deg(aW)$, $b_1=\deg(bW)$, $b_2 = \deg(abW)$, and $d = \lfloor\min((d_1+d_2)/2,(b_1+b_2)/2)\rfloor$.
By the same argument as in the proof of Theorem A, for $\ell=2d$ the space $\mathcal F_z^{\ell,\ell}(\psi_W)$ contains a differential operator $R(z,\partial_z)$ satisfying $\mathcal C_{b_{\psi_W}^{-1}(R)}(f,g,r) = 0$ and $\mathcal C_{R(\xi,\partial_\xi)}(f,g,it) = 0$.  By Proposition \ref{commuting prop}, it contains a nonconstant differential operator $\wt R_{r,t}(z,\partial_z)$, depending on the finite endpoints of the paths $\Gamma_1,\Gamma_2$, which commutes with our integral operator.
In fact, from the vanishing requirements for the concomitants, we see we can take $\wt R_{r,t}(z,\partial)$ to be the 90-degree rotation of $R_{r,t}(z,\partial_z)$, as in \eqref{90 degree rotation}.

\item
Since $\psi_W^\dag(x,z) = \ol{\psi_W(\ol x, \ol z)}$, we have for real $x,\xi$ that
\begin{align*}
U_W(f)(z)
  &= \int_r^\infty\int_t^\infty \psi_W(x,iz)\psi_W^\dag(x,-i\xi)f(\xi)d\xi dx\\
  &= \int_r^\infty\int_t^\infty \psi_W(x,iz)\ol{\psi_W(x,i\xi)}f(\xi)d\xi dx\\
  &= E_WE_W^*(f)(z).
\end{align*}
\end{enumerate}
\end{proof}

\subsection{Simultaneous reflectivity and commutativity for families}
Our exposition so far has revealed the existence of differential operators commuting with integral operators associated to a single rank $1$ bispectral meromorphic function.
We next treat finite families of integral operators corresponding to points of  $\Grad$.

\begin{proof}[Proof of Theorem D]
For each $n$, the point $W_n\in\Grad$ satisfies $\mu_n(z)W_n\subseteq\bbc[z]$, $\mu_n(z)aW_n\subseteq\bbc[z]$, $\mu_n(z)\bbc[z]\subseteq aW_n$, and $\mu_n(z)\bbc[z]\subseteq W_n$, for some polynomial $\mu_n(z)$.  Setting $\mu(z) = \prod_n\mu_n(z)$, we have $\mu(z)\bbc[z]\subseteq W_n$ and $\mu(z)W_n\subseteq \bbc[z]$.
It follows that for all $n$
$$\mathcal F_z(\psi_{W_n}) = \mathcal{D}(W_n)\supseteq \mu(z) \mathcal F_z(\psi_{\exp})\mu(z).$$
Furthermore each $\psi_{W_n}(x,z)$ is of the form
$$\psi_{W_n}(x,z) = \frac{1}{p_n(x)q_n(z)}P_n(x,\partial_x)\cdot \psi_{\exp}(x,z) = \frac{1}{p_n(z)q_n(z)} Q_n(z,\partial_z)\cdot \psi_{\exp}(x,z)$$
for some polynomials $\wt p_n(x),p_n(x),\wt q_n(z),q_n(z)$ and differential operators $\wt P_n(x,\partial_x)$, $P_n(x,\partial_x)$, $\wt Q_n(z,\partial_z)$, and $Q_n(z,\partial_z)$ with polynomial coefficients satisfying
$$\wt P_n(x,\partial_x)\frac{1}{p_n(x)\wt p_n(x)}P_n(x,\partial_x) = q_n(\partial_x)\wt q_n(\partial_x),\ \ \wt Q_n(z,\partial_z)\frac{1}{q_n(z)\wt q_n(z)}Q_n(z,\partial_z) = p_n(\partial_z)\wt p_n(\partial_z).$$
Note that the order of $P_n(x,\partial_x)$ must be the degree of $q_n(z)$, which counts the $z$-poles of $\psi_{W_n}(x,z)$ and therefore is at most $\deg(\mu)$.
Moreover, $\mu(z)/q_n(z)=\theta_n(z)$ and $\mu(z)/\wt q_n(z)=\wt\theta_n(z)$ for some $\theta_n(z),\wt\theta_n(z)\in\bbc[z]$.
Therefore for any $R(z,\partial_z)\in\mathcal F_z(\psi_{\exp})$ we have
$$b_{\psi_{W_n}}^{-1}\left(\mu(z)R(z,\partial_z)\mu(z)\right) = \frac{1}{p_n(x)}P_n(x,\partial_x)\wt\theta_n(\partial_x)b_{\psi_{\exp}}^{-1}\left(R(z,\partial_z)\right)\theta_n(\partial_x)\wt P_n(x,\partial_x)\frac{1}{\wt p_n(x)}.$$

Let $\wt R(z,\partial_z) = \mu(z)(\partial_z-r)^{2\deg(\mu)}R(z,\partial_z)(\partial_z-r)^{2\deg(\mu)}\mu(z)$ so that
$$b_{\psi_W}^{-1}\left(\wt R(z,\partial_z)\right) = D(x,\partial_x)b_{\psi_{\exp}}^{-1}(R(z,\partial_z))\wt D(x,\partial_x)$$
for the operators 
$$D(x,\partial_x) = \frac{1}{p(x)}P(x,\partial_x)\theta(\partial_x)(x-r)^{2\deg(\mu)}\ \ \text{and}\ \ \wt D(x,\partial_x) = (x-r)^{2\deg(\mu)}\theta(\partial_x)\wt P(x,\partial_x)\frac{1}{\wt p(x)}.$$

Clearly the bilinear concomitants $\mathcal C_D(\cdot,\cdot;s)$ and $\mathcal C_{\wt D}(\cdot,\cdot;s)$ are both zero and therefore for all functions $f,g$ analytic in a neighborhood of $r$
\begin{align*}
\mathcal C_{b_{\psi_W}^{-1}\wt R}(f,g;r)
  & = C_{\mathcal Db_{\psi_{\exp}}^{-1}R\wt D}(f,g;r)\\
  & = \mathcal C_{D}(b_{\psi_{\exp}}^{-1}R\wt D\cdot f,g;r) + \mathcal C_{b_{\psi_{\exp}}^{-1}R}(\wt D\cdot f,D^\dag\cdot g;s) + \mathcal C_{\wt D}(f,(b_{\psi_{\exp}}^{-1}R)^\dag D^\dag\cdot g;r)\\
  & = \mathcal C_{b_{\psi_{\exp}}^{-1}R}(\wt D\cdot f,D^\dag\cdot g;r)
\end{align*}
and in in particular $\mathcal C_{b_{\psi_W}^{-1}\wt R}(\cdot,\cdot;r)$ is identically zero if $\mathcal C_{b_{\psi_{\exp}}^{-1}R}(\cdot,\cdot;r)$ is identically zero.
The dimension of the space of formally $\dagger$-symmetric operators in 
$$\mu(z)^N(\partial_z-r)^{2\deg(\mu)}\mathcal F^{2d,2d}(\psi_{\exp})(\partial_z-r)^{2\deg(\mu)}\mu(z)^N$$
is $(2d+1)(d+1)$ and contains a subspace of dimension at least
$$(2d+1)(d+1)- \frac{1}{2}(d+2\deg(\mu))(d+2\deg(\mu)+1)$$
of operators whose concomitants vanish at $t$ and a further subspace of dimension at least
$$(2d+1)(d+1)- \frac{1}{2}(d+2\deg(\mu))(d+2\deg(\mu)+1) - (d+2\deg(\mu))(d+2\deg(\mu)+1)$$
of operators which map under $b_{\psi_W}^{-1}$ to operators whose concomitants vanish at $r$.
Moreover, by the above this condition of vanishing may be taken independently of the choice of $W$.
Thus for $d$ sufficiently large there exists a self adjoint $R(z,\partial_z)$ which reflectively commutes with each of the operators $T_{W_n}$ as stated by the theorem.

Part (b) of the theorem is proved in a similar fashion.
\end{proof}

\section{Examples and higher level Bessel functions}\label{sect:examples}
In this section we give a detailed description of the Fourier algebras of certain higher Bessel functions. We then 
illustrate Theorems A-D by examples in which we show how one writes the reflected and commuting 
differential operators in simple forms in terms of these Fourier algebras.

The subvariety of $\Grad$ corresponding to subspaces $C\subseteq\mathscr C$ supported at $0$ (recall \S \ref{2.1}) is precisely 
the set of solutions of the KP hierarchy whose tau functions are polynomial. Such a $W \in \Grad$ can be viewed as a point 
in a family by letting the scalars in the definition of $C$ vary. The wave function $\psi_W(x,z)$ can then be expressed as a combination of 
Bessel wave functions with coefficients that are rational functions in $x$ by \cite[Prop. 5.1(a) and Eq. (5.4)]{BHY97}.
Then Theorem D can be used to obtain a reflected/commuting differential operator for all integral operators in this family. 

\subsection{First order reflected and commuting differential operators}
The most immediate application we have is related to the basic case: when the bispectral function is $\psi_{\exp}(x,z) = e^{xz}$. Thanks to Theorems A, B and C, we can obtain reflecting and commuting operators of order one.
All previous examples in the literature have involved differential operators of order two or higher.

Consider $W = \bbc[z] \in \Grad$ with wave function $\psi_{\exp}(x,z) = e^{xz}.$ Of course, we have  $\psi^\dag_{\exp}(x,z) = e^{xz}$ and therefore the integral operator 

\[
(T_W f)(z)=
({\mathcal{E}}_W {\mathcal{E}}^*_W f)(z) =  \int_t^\infty K(z,\xi)f(\xi)d\xi,\quad \text{where}\quad K(z,\xi) = \frac{e^{-r(z+\xi)}}{z+\xi},
\]
reflects an operator of order $\le 2$  (see Theorems A and B).
In fact, the operator
$$R_{r,t}(z,\partial_z) = (z-t)\partial_z - rz$$
of order one is reflected by $T_W$.
This follows from Proposition \ref{reflecting prop} and the fact that it satisfies $\mathcal C_{b_\psi^{-1}R(x,\partial_x)}(f,g;r) = 0$ and $\mathcal C_{R(z,\partial_z)}(f,g;t) = 0$ for all smooth functions $f(x),g(x)$ in a neighborhood of $r$ and $t$ (note that that $b_{\psi_{\exp}}^{-1}R(z,\partial_z)= (x-r)\partial_x - xt$).

As a consequence, we should expect that $T_W$ will map eigenfunctions of $R(w,\partial_w)$ to those of $R(-z,-\partial_z)$.
One easily verifies that the function $\varphi_{r,t}(w,\lambda) = e^{rz}(t-w)^{rt+\lambda}$ is an eigenfunction for $R_{r,t}(w,\partial_w)$ with eigenvalue $\lambda$.
Furthermore, for $\lambda$ real with $-1 < rt+\lambda < 0$, a residue integral-type argument shows
$$T_W(\varphi_{r,t}(w,\lambda)) = \varphi_{r,t}(-z,\lambda)\frac{\pi}{\cos(\pi(rt+\lambda))}.$$
The function $\varphi_{r,t}(-z,\lambda)$ is an eigenfunction of $R_{r,t}(-z,-\partial_z)$ with eigenvalue $\lambda$, as expected.

Now, thanks to Theorem C, we have that the integral operator given by 
\begin{align*}
(U_W f)(z)=&
({{E}}_W {{E}}^*_W f)(z) =  \int_r^\infty e^{ixz}\left(\int_t^\infty e^{-ix\xi}   f(\xi)d\xi\right) dx 
\\ 
=&\int_t^\infty \left(\int_r^\infty e^{-ix(\xi-z)} dx \right)    f(\xi)d\xi=:
\int_t^\infty k_r(z,\xi)    f(\xi)d\xi 
\end{align*}
 commutes with the operator $\wt R_{r,t}(z,\partial_z) = R_{r,it}(iz, -i\partial_z) = (z-t)\partial_z - irz$.
 
 \medskip 
 The kernel $k_r(z,\xi)$ has to be interpreted as a distribution acting on a $C^\infty$ rapidly decreasing function $f(\xi)$; in the simplest case when $r=-\infty$ we have 
 $k_{-\infty}(z,\xi)=\int_{-\infty}^\infty e^{-ix(\xi-z)} dx=2\pi \delta (\xi - z)$. For finite values of $r$ the integral can be written as $i(1-e^{-ir(\xi-z)})/(\xi-z)$ 
 plus the distribution given by $\int_{0}^\infty e^{-ix(\xi-z)} dx,$ i.e., the Fourier transform  $\hat H$ of the Heaviside function $H$ 
 in the variable $\xi-z$. One has $$\hat H(z)=\pi \delta (z) - i p.v.(z^{-1}),$$
see \cite[p. 360, Formula (23)]{GS64} and \cite [Example 7.1.17]{H03}. Hence, $\int_t^\infty k_r(z,\xi) f(\xi)d\xi$ can be written as a linear combination of a simple term and the integral form $0$ to $\infty$ of the Fourier transform of $f$.
\bigskip

Now since $W$ is fixed by the involution $ac$, we know that the usual adjoint $\wt R_{r,t}^*(z,\partial_z) = \overline{\wt R_{r,t}^\dag(z,\partial_z)}$ will commute with $U_W$ also.  Indeed $\wt R_{r,t}(z,\partial_z)^* = -\wt R_{r,t}(z,\partial_z)-1$.
Consequently the self-adjoint differential operator
$$-\wt R_{r,t}(z,\partial_z)^2-\wt R_{r,t}(z,\partial_z) = -\partial_z(z-t)^2\partial_z + ir\{z(z-t),\partial_z\} + r^2z^2$$
commutes with the self-adjoint integral operator $U_W$. Notably, it has both real and imaginary pieces.

\subsection{Higher Bessel Functions}\label{sec:higher bessel}
For any $\lambda=(\lambda_1,\lambda_2,\dots,\lambda_{n})\in\mathbb R^n$ consider the generalized Bessel operator 
\begin{equation}\label{Pa}
L_\lambda(z,\partial_z)= z^{-n}\,(z\partial_z -\lambda_1)(z\partial_z-\lambda_2)\dots(z\partial_z-\lambda_{n}).
\end{equation}
An eigenfunction of  this operator is given by the generalized bispectral Bessel function $\psi_{\lambda}(x,z)$, as defined in \cite[Definition 1.6]{BHY97}, 
which satisfies $L_\lambda(x,\partial_x)\cdot\psi_\lambda(x,z) = z^n\psi_\lambda(x,z)$. In this section, we explicitly identify all 
the Bessel operators $L_\lambda$ that are related via Darboux transforms to the exponential operator $\partial_z^n$ (see \eqref{orbit}) and we give 
an explicit expression for the corresponding bispectral functions $\psi_\lambda(x,z)$ in \eqref{bessel function equation} (see also \eqref{lambda} and \eqref{ulambda}). 
Furthermore, in Theorem \ref{Fourier-Bessel} 
we describe in detail the left and right Fourier algebras, with their bifiltrations, of every $\psi_\lambda(x,z)$. 
These descriptions are used in the examples of Sections \ref{KdV1} and \ref{22}.

First, let us observe that $L_\lambda=L_{w \lambda}$ for any permutation $w \in S_n$, so we can rewrite it as 
$$L_\lambda(z,\partial_z)= z^{-n+1}(z\partial_z -\lambda_ 1+1)\dots(z\partial_z-\lambda_{j-1}+1)(z\partial_z-\lambda_{j+1}+1)
(z\partial_z-\lambda_{n}+1) z^{-1}(z\partial_z-\lambda_{j}).$$
By performing a bispectral Darboux transformation we obtain a new operator
$$ z^{-1}(z\partial_z-\lambda_j)z^{-n+1}(z\partial_z -\lambda_ 1+1)\dots(z\partial_z-\lambda_{j-1}+1)(z\partial_z-\lambda_{j+1}+1)
(z\partial_z-\lambda_{n}+1)  =L_{\lambda+r_j}(z,\partial_z)$$
for $r_j=ne_j-(1,1,\dots, 1)$, where $\{e_1,\dots, e_n\}$ is the canonical basis of $\mathbb C^n$.
Hence, for two given operators $L_\alpha$ and $L_\beta$ of the form \eqref{Pa}, one can be obtained from the other by performing iterations of Darboux transformations if and only if $\beta=w \alpha+\lambda$, for $w \in S_n$ and $\lambda$ in the lattice  generated by $\{r_j\}_{j=1}^n$
$${P}=\{\lambda\in\mathbb Z^n: \lambda_1+\dots+\lambda_n=0, \quad \lambda_j-\lambda_{j+1}\in n\mathbb Z \quad \forall j=1,\dots,n-1\}.$$
Note that this is exactly the lattice of weights of type $A_{n-1}$, up to a scaling factor of $n$ with respect to the usual convention: we are considering the simple roots $\{\alpha_j:=n(e_j-e_{j+1})\}_{j=1}^{n-1}$. Hence, $L_\alpha$ and $L_\beta$ are related via Darboux transforms if and only if $\alpha$ and $\beta$ are in the same $W$-orbit, where $W= S_n\ltimes {P}$ is the extended affine Weyl group of type $A_{n-1}$. Here we are considering the action of the group $W$ extended to $\mathbb R^n$ rather than its restriction to the $(n-1)$-dimensional space $E$ generated by $\{\alpha_j\}$,

The fundamental alcove is defined as
$$A=
\left\{\lambda\in E:0\le\langle \lambda,\alpha_j\rangle ,\quad 
\langle \lambda,\textstyle \sum_{j=1}^{n-1}\alpha_j\rangle\le n
\right\}.
$$
The stabilizer of $A$ in $W$ is the non-trivial subgroup $\Omega=\{w \in W: w A = A\}$.
We are specially interested in the weight $\lambda_{\exp}=(0,1,\dots,n-1)$, for which $L_{\lambda_{\exp}}=\partial^n$. The corresponding point in $E$, given by $\lambda_{\exp}-(n-1) (1,\dots,1)$, is the unique point in $A$ stable by $\Omega$.

 The discussion above shows that, for any $w \in W$,  $\psi_{w \lambda_{\exp}}(x,z)$ is a bispectral Darboux transformation from the exponential function $\psi_{\lambda_{\exp}}(x,z) = e^{xz}$.
We can now describe the orbit $W\lambda_{\exp}$ as the set of elements $\lambda=(\lambda_1,\dots,\lambda_n)\in \mathbb Z^n$ such that 
\begin{equation}\label{orbit}
\lambda_1+\dots+\lambda_n=n(n-1)/2\quad \text{and}\quad 
\{(\lambda_j-\min_{1\le k\le n}{\lambda_k}) \mod n\}_{j=1}^{n}=\{0,1,\dots,n-1\}.
\end{equation}
Furthermore, for any $\lambda=(\lambda_1,\dots,\lambda_n)\in W\lambda_{\exp}$
we have 
\begin{equation}\label{lambda}
\lambda=w_\lambda\lambda_{\exp}+
\sum_{j=1}^{n} r_jc_j,\quad {\text{ with }} c_j:=
\left\lfloor{\left(\lambda_j-\min_{1\le k \le n}\{\lambda_{k}\}\right)\over n}\right\rfloor 
{\text{ and }} w_\lambda\in S_n. 
\end{equation}
More specifically, $w_\lambda$ is the permutation that sends $j-1$ to $\left((\lambda_j-\min_{1\le k\le n}{\lambda_k}) \mod n\right)$ for $j=1,\dots, n$.
This shows us that the degree $d$ of $\psi_\lambda(x,z)$, as defined in Sect. \ref{sec:Wilson}, is $d=c_1+\dots+c_n$.

\begin{figure}[h]
\begin{tikzpicture}[ scale=1]
 \usetikzlibrary{arrows,decorations.markings}

 \filldraw[
  top color=green!50,white]
(0,0) -- ({cos(30+0)},{sin(30+0)}) -- ({cos(30+60)},{sin(30+60)}) -- (0,0);

 \draw [ blue,thin,decoration={markings,mark=at position 1 with {\arrow[scale=2,>=stealth] {>}}},postaction={decorate}]  (0,0) -- (1.73,0);
 \draw [ blue,thin,decoration={markings,mark=at position 1 with {\arrow[scale=2,>=stealth] {>}}},postaction={decorate}]  (0,0) -- ({1.73*cos(30+90)},{1.73*sin(30+90)});
  
            \foreach \y in {0,1,2,3,4,5} {
                        \foreach \x in {0,1,2} {
                        \foreach \z in {0,1,2} {
                       \fill[black]  {({cos(30+\z*120)+cos(30+\x*120)},{sin(30+\z*120)+sin(30+\x*120)})+({0.577*cos(30+\y*60+90)},{0.577*sin(30+\y*60+90)})} circle [radius=2pt];  
        }}}
           \foreach \y in {0,1,2} {
                        \foreach \x in {0,1,3,4} {
                       \fill[orange]  ({4*0.577*cos(\y*120)},{4*0.577*sin(\y*120)})+
({0.577*\x*cos((\y-1)*120)},{0.577*\x*sin((\y-1)*120)}) circle [radius=2pt];  
        }}
 
            \foreach \y in {0,1,2,3,4,5} {
                        \foreach \x in {0,1,2} {
                       \fill[red]  ({cos(30+\x*120)},{sin(30+\x*120)})+({0.577*cos(30+\y*60+90)},{0.577*sin(30+\y*60+90)}) circle [radius=2pt];  
        }}
            \foreach \y in {0,1,2,3,4,5} {
                       \fill[cyan]  ({0.577*cos(30+\y*60+90)},{0.577*sin(30+\y*60+90)}) circle [radius=2pt];  
        }


 \foreach \y in {0,1,2,3,4,5} {
  \foreach \x in {0,1,2,3} {
             \draw [dotted,-, thin] ({3*cos(30+\y*60)},{3*sin(30+\y*60)})
             +({-\x*cos(30+\y*60+60)},{-\x*sin(30+\y*60+60)}) 
             -- ({-3*cos(30+\y*60)+\x*cos(30+\y*60-60)},{\x*sin(30+\y*60-60)-3*sin(30+\y*60)})
             ;
}}
 


\end{tikzpicture}
 \caption{}\label{An}
\end{figure} 
 
In Figure \ref{An}, we have the simple roots of $A_2$ in blue. The base point corresponding to $\lambda_{\exp}$ and its permutations are in cyan.
The fundamental alcove, containing the base point, is in green.
The possible one-step, two-step, and three-step Darboux transformations of the base point are in red, black, and orange, respectively.
 
Since for any $\lambda$ we have $L_\lambda=L_{w_\lambda\lambda}$ we can assume from the start that $w_\lambda$ is trivial to ease the technical notation for the computation of the Darboux transformation. Hence, for $\lambda\in P$ of the form in \eqref{lambda} with $w_\lambda=1$, we define the set of \emph{exceptional degrees} to be 
\begin{equation}\label{ulambda}
U_{\lambda} = \bigcup_{j=0}^{n-1}\{j,j+n,j+2n\dots,j+(c_j-1)n\}.
\end{equation}
Then,  $\psi_{\lambda}(x,z)$ is the bispectral Darboux transformation of degree $d$ from the exponential function $\psi_{\lambda_{\exp}}(x,z) = e^{xz}$ explicitly given by 
\begin{equation}\label{bessel function equation}
\psi_{\lambda}(x,z) = \frac{1}{x^dz^d}\prod_{k\in U_{\lambda}} (z\partial_z - k)\cdot e^{xz}\quad \text{with}\quad d=|U_{\lambda}|.
\end{equation}
%
%
%
The corresponding point in the adelic Grassmannian $\Grad$ is
$$W_{\lambda} = \frac{1}{z^d}\vspan\{z^k: 0\leq k\notin U_\lambda\}.$$
It follows from \eqref{bessel function equation} that $\psi_{\Be(\nu)}(x,z) = \psi_{(-\nu,\nu+1)}(x,z)$ for the Bessel functions defined in \eqref{bispectral bessel}.
The next theorem describes the left and right Fourier algebras of generalized Bessel functions.

\begin{theorem}\label{Fourier-Bessel}
Consider the generalized Bessel function $\psi_\lambda(x,z)$.
Since $\psi_\lambda(x,z) = \psi_\lambda(z,x)$, the left and right Fourier algebras are isomorphic via
$$\mathcal F_x(\psi_\lambda)\rightarrow\mathcal F_z(\psi_\lambda):\ \ L(x,\partial_x)\mapsto L(z,\partial_z).$$
Furthermore
\begin{equation}\label{bessel fourier algebra}
\mathcal F_z(\psi_\lambda)
  = \vspan_\bbc\left\lbrace z^r f(z\partial_z): r\in\bbz, f(\lambda)\in\bbc[\lambda],\ \text{and}\ f(n-d) = 0\ \forall\ \substack{n\in\bbz_+\diff U_\lambda\\ n+r\notin \bbz_+\diff U_\lambda}\right\rbrace
\end{equation}
and the bifiltration satisfies
\begin{equation}\label{bessel fourier algebra bifiltration}
\mathcal F_x^{\ell,m}(\psi_{\lambda}) =  \vspan\{x^{r}f(x\partial_x)\in \mathcal F_x(\psi_{\lambda}): \deg(f)\leq \ell,\ \deg(f)+r\leq m\}.
\end{equation}
\end{theorem}
\begin{proof}
First note that the right Fourier algebra $\mathcal F_z({\psi_{\lambda}})$ contains $S(z,\partial_z):= z\partial_z$.  The action of $\ad_{S}$ on $\mathcal F_z({\psi_{\lambda}})$ preserves $\mathcal F_z^{\ell,m}({\psi_{\lambda}})$ for all integers $\ell,m$.
A differential operator with rational coefficients $R(z,\partial_z)$ is an eigenfunction of $\ad_S$ if and only if $R(z,\partial_z) = z^rf(S(z,\partial_z))$ for some polynomial $f(\lambda)\in\bbc[\lambda]$.
By choosing a basis for $\mathcal F_z^{\ell,m}({\psi_{\lambda}})$ consisting of eigenvectors of $\ad_S$, we see that $\mathcal F_z^{\ell,m}({\psi_{\lambda}})$ is spanned by differential operators of this form.
Thus so too is $\mathcal F_z({\psi_{\lambda}}).$

Now if we recall that $W_{\lambda} = \frac{1}{z^d}\vspan\{z^k: \bbz_+\diff U_\lambda\}$ and that
$$\mathcal F_z({\psi_{\lambda}}) = \vspan_\bbc\{z^rf(S(z,\partial_z)): r\in\bbz,\ f(\lambda)\in\bbc[\lambda],\ z^rf(S(z,\partial_z))\cdot W_\lambda\subseteq W_\lambda\},$$
the expression \eqref{bessel fourier algebra} follows immediately.

The generalized Fourier map sends $x\partial_x$ to $z\partial_z$.  Since $\ad_S(z^rf(S(z,\partial_z))) = rz^rf(S(z,\partial_z))$, the preimage of $z^rf(S(z,\partial_z))\in\mathcal F_z({\psi_{\lambda}})$ under $b_{\psi_{\lambda}}$ must be an eigenfunction of of $\ad_S$ with eigenvalue $-r$, i.e. of the form $x^{-r}g(S(x,\partial_x))$.

The generalized Fourier map sends a differential operator with leading coefficient $x^j\partial_x^k$ to a differential operator with leading coefficient $z^k\partial_z^j$, and consequently the bifiltration is given by \eqref{bessel fourier algebra bifiltration}.
\end{proof}

\subsection{The step one KdV case, \cite{DG86}}\label{KdV1}
Consider the $\tau$-function $\tau(x) = x^3 - 3t_3$.
The wave functions of the KP flow corresponding to this $\tau$-function classify the first step bispectral functions in the KdV family of \cite[Eq. 1.39]{DG86}, and are given by
\begin{align*}
\psi(x,z)
  &= \frac{(x-z^{-1})^3 +z^{-3} - 3t_3}{x^3 - 3t_3}e^{xz}.
\end{align*}
We can rewrite this as
$$\psi(x,z) = \frac{x^3}{x^3-3t_3}\psi_{(-2,3)}(x,z) + \frac{-3t_3}{x^3-3t_3}\psi_{(0,1)}(x,z).$$
The function $\psi_{(0,1)}(x,z) = \psi_{\lambda_{\exp}}(x,z)$ is the base point of the level two lattice of Bessel functions in Sect. \ref{sec:higher bessel}, and the function $\psi_{(-2,3)}(x,z)$ is a one-step Darboux transformation from it.

Let us consider the integral operator $T_\psi$ given by the kernel 
$$
K_\psi(z,\xi) := \int_r^\infty \psi(x,-z)\psi^\dag(x,-\xi) dx = e^{-r(z+\xi)}\frac{(r +z^{-1} +\xi^{-1})^3-3t_3-z^{-3}-\xi^{-3}}{(r^3-3t_3)(z+\xi)}\cdot
$$
By theorem A we know that $T_\psi$ reflects a differential operator. We explicitly build such operator, which is in $\mathcal F_z(\psi_{(0,2)})\cap \mathcal F_z(\psi_{(-2,3)})$, and is therefore independent of $t_3$. Note that Theorem D guarantees the existence  of an operator in this intersection that is reflected by $T_\psi$, and in particular by the kernels associated to $\psi_{(0,2)}$ and  $\psi_{(-2,3)}$, which are limiting cases of $K_\psi$.
By Theorem \ref{Fourier-Bessel} we can write it in the form
$$R(z,\partial_z) = \sum_{r=-4}^5 z^r f_r(z\partial_z),$$
for some polynomials $f_r(u)\in\bbc[u]$ which satisfy $f_r(n-2)=0$ whenever $n\in\{0,2,4,5,6,\dots\}$ and $n+r\notin\{0,2,4,5,6,\dots\}$.
If we further require our operator to be self-adjoint, we can require the polynomials to satisfy the additional symmetry $f_r(u) = f_r(-u-1-r)$.
A reflected operator whose polynomials satisfy this condition is specifically calculated to be
\begin{align*}
f_{-4}(u) &= - it^{3} u (u - 5) (u - 3) (u - 2) (u - 1) (u + 2),\\
f_{-3}(u) &= 3 it^{2} u (u - 4) (u - 2) (u - 1)^2 (u + 2)\\
f_{-2}(u) &= - (3/2) it u (u - 3) (u - 1) (u + 2) (- r^{2} t^{2} + 2 u^{2} - 2 u),\\
f_{-1}(u) &=  iu^{2} (u - 2) (u + 2) (- 6 r^{2} t^{2} + u^{2} - 1),\\
f_0(u) &= (3/2) ir^{2} t u (u + 1) (5 u^{2} + 5 u - 8),\\
f_1(u) &= - 3 ir^{2} u (u + 2) (- r t + u + 1) (r t + u + 1),\\
f_2(u) &= - (1/2) ir^{4} t (r^{2} t^{2} + 12 u^{2} + 36 u + 18),\\
f_3(u) &= 3 ir^{4} (u + 2)^{2}, \quad f_4(u) = (3/2) ir^{6} t, \quad f_5(u) = -ir^6.
\end{align*}
After a change of variables $z\mapsto iz$, $t\mapsto it$, $\partial_z\mapsto-i\partial_z$, we obtain an operator which commutes with the integral operator given by
$$f(\xi)\mapsto \int_r^\infty \int_t^\infty \psi(x,i z)\ol{\psi(x,i\xi)}f(\xi)d \xi dx.$$
Specifically, we obtain the self-adjoint differential operator (with respect to \eqref{adjoint formula})
$$\wt R_{r,t}(z,\partial_z) = R_{r,it}(iz,-i\partial_z) = \sum_{j=0}^3 \partial_z^k a_k(z)\partial_z^k$$
with  
\begin{align*}
a_0(z) & = \frac{1}{6}z^2(3t^3r^6-54tr^4) + z^5r^6 - \frac{3tz^4r^6}{2} + 12z^3r^4,\\
a_1(z) & = (z-t)\left(3z^4r^4-3tz^3r^4+12z^2r^2+9tzr^2-9r^2t^2\right),\\
a_2(z) & = (z-t)^2\left(3z^3r^2-\frac{3}{2}tz^2r^2 + 12t\right),\quad a_3(z) = (z-t)^3z^2.
\end{align*}

However, if we do not mind a dependence on the parameter $t_3$ then we can get a differential operator of much smaller order.
Applying Theorem B, we obtain that $T_\psi$ reflects a differential operator of order three, namely the operator
\begin{align*}
&R_{r,t}(z,\partial_z)
  = -(z-t)^2z\partial_z^3 + \big(rt^3 - 3rtz^2 + 2rz^3 - 3t^3t_3z^2 + 3t^2t_3z^3 - \frac{3}{2}t^2 + 6tz - \frac{9}{2}z^2 \big)\partial_z^2\\
  & \big(r^2t^2z - r^2z^3 + 6rt^3t_3z^2 - 6rt^2t_3z^3 - 6rtz + 6rz^2 - 6t^3t_3z + 9t^2t_3z^2 + 6t^2z^{-1} - 6t - z \big)\partial_z\\
  & -6rt^3z^{-2} - 3t^2z^{-2} + r^3tz^2 - 3r^2t^3t_3z^2 + 3r^2t^2t_3z^3 - \frac{3}{2}r^2z^2 + 6rt^3t_3z - 9rt^2t_3z^2 + 3t^2t_3z + 3.
\end{align*}


%
%
%
%
%
%
 
\subsection{The Young diagram (2,2)}\label{22}
Consider the Schur tau function
$$\tau(x) = x^4-12 x t_3+12 t_2^2$$
corresponding to the Young diagram (2,2).

The wave functions of the associated KP flow are
$$\psi(x,z) = e^{xz}\frac{(x-z^{-1})^4-12 (x-z^{-1}) (t_3-z^{-3}/3)+12 (t_2-z^{-2}/2)^2}{x^4-12 x t_3+12 t_2^2}\cdot$$
This can be rewritten as
\begin{align*}
\psi(x,z)
  & = \frac{x^4-2t_2x^2}{\tau(x)}\psi_{(-1,-2,4,5)}(x,z)\\
  & + \frac{8t_2x^2-12xt_3}{\tau(x)}\psi_{(-1,1,2,4)}(x,z) + \frac{12t_2^2-6x^2t_2}{\tau(x)}\psi_{(0,1,2,3)}(x,z).
\end{align*}
The third wave function on the right hand side is the base point $\psi_{\lambda_{\exp}}(x,z)$ of the lattice of level four Bessel functions in Sect. \ref{sec:higher bessel}, while the second and first wave functions are one-step and two-step Darboux transformations from it, respectively.

As in previous examples, we consider the integral operator $T_\psi$ with kernel 
\begin{align*}
K_\psi(z,\xi)
  :&= \int_r^\infty \psi(x,-z)\psi^\dag(x,-\xi) dx\\
  &= e^{-r(z+\xi)}\frac{\left(r+\frac{1}{z}+\frac{1}{\xi}\right)^4 - 12\left(r+\frac{1}{z}+\frac{1}{\xi}\right)\left(t_3 + \frac{1}{3z^3}+\frac{1}{3\xi^3} \right) + 12\left(t_2 - \frac{1}{2z^2}+\frac{1}{2\xi^2}\right)^2}{(z+\xi)(r^4-12rt_3+12t_2^2)}.
\end{align*}
By Theorem A we know that $T_\psi$ reflects a differential operator. We explicitly build such operator, which is in the intersection
$$\mathcal F_z(\psi_{(-1,-2,4,5)})\cap \mathcal F_z(\psi_{(-1,1,2,4)})\cap \mathcal F_z(\psi_{(0,1,2,3)}),$$
and is therefore independent of $t_2$ and $t_3$. Note that its existence is guaranteed by Theorem D.
It will neccessarily be reflected by the integral operators with kernels associated to $\psi_{(-1,-2,4,5)}$, $\psi_{(-1,1,2,4)}$ and  $\psi_{(0,1,2,3)}$, appearing as limiting cases of $K_\psi$.
By Theorem \ref{Fourier-Bessel}, we can write it in the form
$$R(z,\partial_z) = \sum_{r=0}^6 z^r f_r(z\partial_z)$$
for some polynomials $f_r(u)\in\bbc[u]$.
The constructed reflected operator is order four and the polynomials are given by
\begin{align*}
f_0(u)    &= t^2u(u^3 + 2u^2 - u - 2), \\
f_1(u)    &= -2tu(u + 1)(4rt + u^2 + u(2rt + 3) + 2), \\
f_2(u)    &= (u + 1)(u + 2)(6r^2t^2 + 12rt + u^2 + u(8rt + 3) + 2), \\
f_3(u)    &= -4r(u + 2)(r^2t^2 + 6rt + u^2 + u(3rt + 4) + 4), \\
f_4(u)    &= r^2(r^2t^2 + 20rt + 6u^2 + 2u(4rt + 15) + 38), \\
f_5(u)    &= -2r^3(rt + 2u + 6), \quad f_6(u)  = r^4.
\end{align*}
The corresponding fourth order differential operator
\begin{align*}
\wt R_{r,t}(z,\partial_z) &= R_{r,it}(iz,-i\partial_z) = -\partial_z^2(z-t)^2z^4\partial_z^2\\
  & + 2ri\{\partial_z^3,z^4(z-t)^2\}  + 2\partial_zz^3(3r^2z(z-t)^2 - 4(z-t))\partial_z\\
  & - 2ri\{\partial_x,z^2((r^2z^2+12)(z-t)^2 - 12tz + 14z^2)\}\\
  & + z^2(-r^4z^2(z-t)^2 + 2r^2(6t^2-24tz + 19z^2) - 4).
\end{align*}
is clearly self-adjoint and by Theorem C, it commutes with the integral operator given by
$$U_\psi: f(\xi)\mapsto \int_r^\infty \int_t^\infty \psi(x,i z)\psi^\dag(x,-i\xi)f(\xi)d \xi dx.$$
\medskip
\\
\noindent
{\bf Acknowledgements.} The research of W.R.C. was supported by a 2018 AMS-Simons Travel Grant. M.Y. was supported by NSF grant DMS-1901830 and
Bulgarian Science Fund grant DN02/05. The research of I.Z. was supported by CONICET PIP grant 112-200801-01533, SECyT grant 30720150100255CB
and a Fulbright scholarship. I.Z. and F.A.G. have also benefitted from conversations during the Research in Pairs program at Oberwolfach in the fall of 2018.

\bibliographystyle{plain}

\end{document}